\documentclass{article}
\usepackage{authblk}
\usepackage{natbib}
\usepackage{amsthm, amsmath, amsfonts, amssymb}
\usepackage{graphicx} 

\usepackage{epsfig}

\usepackage{makecell}
\usepackage{array,multirow}

\usepackage{color}
\usepackage[normalem]{ulem}
\usepackage{cancel}

\usepackage{hyperref}

\usepackage[ruled, vlined, linesnumbered]{algorithm2e} 
\numberwithin{equation}{section}

\DeclareMathAlphabet{\pazocal}{OMS}{zplm}{m}{n}
\DeclareMathOperator*{\argmin}{arg\,min}
\DeclareMathOperator*{\argmax}{arg\,max}
\DeclareMathOperator*{\essup}{ess\,sup}

\def\Rset{\ensuremath{\mathbb R}} 
\def\Nset{\ensuremath{\mathbb N}} 
\def\Gset{\ensuremath{\mathcal G}} 
\def\S{\ensuremath{\mathcal S}} 
\def\A{\ensuremath{\mathcal A}} 
\def\Br{\ensuremath{\mathcal B}} 
\def\I{\ensuremath{\mathcal I}} 
\def\F{\ensuremath{\mathcal F}} 
\def\C{\ensuremath{\mathcal C}} 
\def\P{\ensuremath{\pazocal P}} 
\def\L{\ensuremath{\mathcal L}} 
\def\D{\ensuremath{\mathbb D}} 

\def\M{\ensuremath{\mathcal M}} 
\def\Pr{\ensuremath{\mathbb P}} 
\def\Q{\ensuremath{\mathbb Q}} 
\def\E{\ensuremath{\mathbb E}} 
\def\V{\ensuremath{\mathbb V}} 
\def\W{\ensuremath{\mathbb W}} 
\def\B{\ensuremath{\mathbb B}} 
\def\Ops{\ensuremath{O_{\Pr^*}}} 

\def\X{\ensuremath{\mathbf X}} 
\def\x{\ensuremath{\mathbf x}} 

\def\r{\ensuremath{\mathbf r}} 
\def\c{\ensuremath{\mathbf c}} 
\def\d{\ensuremath{\mathbf d}} 

\def\ie{\emph{i.e.,} }

\def\mnsr{M(\C_n^s,\r)}
\def\mns{M(\C_n^s)}

\usepackage{array} 

\newcounter{hypocounter}
\renewcommand\thehypocounter{(H\arabic{hypocounter})} 
\newenvironment{hypo}{\refstepcounter{hypocounter}\begin{tabular}{m{.85\linewidth}l}}{&\thehypocounter\end{tabular}}

\newcommand\ind[1]{\mathbf{1}_{#1}}
\providecommand{\keywords}[1]{\textbf{\textit{Keywords: }} #1}

\theoremstyle{plain}
\newtheorem{definition}{Definition}[section]
\newtheorem{proposition}{Proposition}[section]
\newtheorem{theorem}{Theorem}[section]

\newtheorem{corollary}{Corollary}[section]
\newtheorem{example}{Example}

\theoremstyle{remark}
\newtheorem{remark}{Remark}

\title{Consistent Regression using Data-Dependent Coverings}
\author{Vincent Margot, Jean-Patrick Baudry, Frederic Guilloux, Olivier Wintenberger}
\affil{Sorbonne Universit\'e, CNRS, LPSM, F-75005 Paris, France}

\date{}

\begin{document}
	\maketitle
	\begin{abstract}
		We introduce a novel method to generate \emph{interpretable} regression function estimators. The idea is based on   \emph{data-dependent coverings}. The aim is to extract from the data a covering of the feature space instead of a partition. The estimator predicts the empirical conditional expectation over the cells of the partitions generated from the coverings. Thus, such estimator has the same form as those issued from \emph{data-dependent partitioning} algorithms. We give sufficient conditions to ensure the consistency, avoiding the sufficient condition of shrinkage of the cells that appears in the former literature. Doing so, we reduce the number of covering elements. We show that such coverings are \emph{interpretable} and each element of the covering is tagged as \emph{significant} or \emph{insignificant}.
		
		The proof of the consistency is based on a control of the error of the empirical estimation of conditional expectations which is interesting on its own. 
	\end{abstract}
	
	\keywords{Consistency, Nonparametric regression, Rule-based algorithm, Data-dependent covering, Interpretable learning.}

\section{Introduction}
We consider the following regression setting: let $\left(\X, Y\right)$ be a pair of random variables in $\Rset^d\times\Rset$ of unknown distribution $\Q$ such that
\begin{equation*}\label{model}
Y = g^*(\X) + Z,
\end{equation*}
where $\E[Z] = 0$, $\V(Z) = \sigma^2 < \infty$ and $g^*$ is a measurable function from $\Rset^d$ to $\Rset$.\\
We make the following common assumptions:
\begin{itemize}
	\item 
	\begin{hypo}\label{H:indep}
		$Z$ is independent of $\X$;
	\end{hypo}\\
	\item
	\begin{hypo}\label{H:bounded}
		$Y$ is bounded: $\Q(\S)=1$ where $\S=\Rset^d \times [-L, L]$ for some unknown $L>0$.
	\end{hypo}
\end{itemize}
The accuracy of $g:\Rset^d \to \Rset$ such that $\E[g(\X)^2]<\infty$ is measured by its quadratic risk defined as
\begin{equation*}\label{risk}
	\L \left( g\right) := \E \left[( g(\X)-Y)^2\right]\,.
\end{equation*}
Then we have
\begin{eqnarray*}
	g^*(\X) = \E\left[Y | \X\right] = \argmin \limits_{g} \L \left( g\right)  \text{ a.s},
\end{eqnarray*}
where the $\argmin$ ranges over all the measurable functions $g$ with $\E[g(\X)^2]<\infty$.

Given a sample $\mathbf D_n = \left( (\X_1,Y_1),\dots, (\X_n, Y_n) \right)$, we aim at predicting $Y$ conditionally on $\X$. The observations $\left(\X_i,Y_i\right)$ are assumed  independent and identically distributed (i.i.d.) from the distribution $\Q$. 

Let us denote the empirical measure for any $\r \subseteq \Rset^d$ and $I \subseteq \Rset$ 
\begin{equation*}
	\Q_n(\r\times I) := \frac 1n \sum_{i=1}^n \ind{\X_i \in \r} \ind{Y_i \in I}.
\end{equation*}
	For simplicity of notation, for any $\r \subseteq \Rset^d$ we write $\Q_n(\r)$ instead of $\Q_n(\r \times \Rset)$ as well as $\Q(\r)$ instead of $\Q(\r \times \Rset)$.

We consider a set of functions $\Gset_n$. Any $\tilde g_n \in \argmin_{g \in \Gset_n} \L(g)$ minimizes the risk over $\Gset_n$ but is not an estimator since it depends on the knowledge of $\Q$.  
Following the Empirical Risk Minimization (ERM) principle \citep[Section~1.5]{Vapnikbook}, we define the empirical risk and an empirical risk minimizer over $\Gset_n$ as, respectively,
\begin{equation}\label{argmin}
	\L_n(g) := \frac1n\sum_{i=1}^{n} \left( g(\X_i) - Y_i \right)^2 \text{ and } g_n \in \argmin_{g \in \Gset_n} \L_n(g).
\end{equation}

The aim of this paper is to provide \emph{interpretable}\footnote{A formal definition of interpretability convenient to our framework is proposed and discussed in Section~\ref{sec:interpretability}.} learning algorithms that generate weakly consistent empirical risk minimizers $g_n$, i.e. such that their excess of risk $\ell \left( g^*, g_n\right)$ fulfill
\begin{equation*}
	\ell \left( g^*, g_n\right) := \L(g_n) - \L(g^*) = \E[(g_n(\X)-g^*(\X))^2]=o_\Pr(1)\,.
\end{equation*}
Note that $\Pr$ refers to an underlying probability measure of reference that is unique (on the contrary to $\Q$ that is arbitrary) and that $\E$ corresponds to the expectation under $\Pr$. 

\subsection{Rule-based algorithms using partitions and quasi-coverings}\label{sec:rule-based}

In this paper, we consider algorithms generating interpretable models that are rule-based, such as \emph{CART} \citep{CART}, \emph{ID3} \citep{Quinlan86}, \emph{C4.5} \citep{Quinlan93}, \emph{FORS} \citep{Karalivc97}, \emph{M5 Rules} \citep{Holmes99}. In these models, the estimator is explained by the realization of a simple condition, an \emph{If-Then} statement of the form:
\begin{eqnarray}\label{rule_form}
	\text{IF} && (\X\in c_1) \text{ And }(\X\in c_2) \text{ And } \dots \text{ And } (\X\in c_k)\\
	\text{THEN} &&g_n(\X)=p\nonumber
\end{eqnarray}
where  $p\in \Rset$ and each $c_i\subseteq\Rset^d$ is expressed in its simplest shape. For instance, the subset $[a_1, b_1] \times [a_2, b_2] \times \Rset^{d-2}$ will be expressed as two $c_i$'s: $$c_1 = [a_1, b_1] \times \Rset^{d-1} \text{ and } c_2 = \Rset \times [a_2, b_2] \times \Rset^{d-2}.$$

\color{black} The \emph{If} part, called the \emph{condition} of the rule, or simply the rule, is composed of the conjunction of $k$ basic tests which forms a subset of $\Rset^d$ and $k$ is called the \emph{length} of the rule. The \emph{Then} part, called the \emph{conclusion} of the rule, is the estimated value when the rule is \emph{activated} (i.e., when the condition in the \emph{If} part is satisfied). The rules are easy to understand and allow an interpretable decision process when $k$ is small. For a review of the best-known algorithms for descriptive and predictive rule learning, see \cite{Zhao03} and \cite{Furnkranz15}. 

Formally, models generated by such algorithms are defined by a corresponding \emph{data-dependent partition} $\P_n$ of $\Rset^d$. Each element of the partition is named a \emph{cell}. Let us define for any $\r\subseteq\Rset^d$ such that $\Q(\r)>0$ and for any $h$ measurable
\begin{equation*}
	\E[h(Y) \mid \X \in \r]:= \frac{\E[h(Y)\ind{\X\in \r}]}{\Pr(\X\in\r)}
\end{equation*}
and
\begin{equation*}
	\V(Y \mid \X \in \r):= \E \left[ Y^2 \mid \X \in \r \right] - \E \left[ Y \mid \X \in \r\right]^2
\end{equation*}
and their empirical counterparts for any $\r\subseteq\Rset^d$ such that $\Q_n(\r)>0$:
\begin{equation*}
 	\E_n[h(Y) \mid \X \in \r]:= \frac{\sum_{i=1}^{n} h(Y_i) \ind{\X_i \in \r}}{\sum_{i=1}^{n} \ind{\X_i \in \r}}
\end{equation*}
and
\begin{equation*}
	\V_n(Y \mid \X \in \r):= \E_n \left[ Y^2 \mid \X \in \r \right] - \E_n \left[ Y \mid \X \in \r\right]^2.
\end{equation*}


Any such algorithm following the ERM principle associates to $\P_n$ an estimator such that
\begin{equation}\label{estimator_partition}
	g_n \ind{ \bigcup \{ A\in\P_n : \Q_n(A)>0 \}} = \mathop{\sum_{A\in\P_n}}_{\Q_n(A)>0} \E_n[Y \mid \X \in A] \ind A
\end{equation}
which minimizes the empirical risk among the class of piecewise constant functions over $\P_n$.

Those algorithms use the dataset $\mathbf D_n$ twice: The partition $\P_n = \P_n(\mathbf D_n)$ is constructed according to the dataset, then the dataset is used again to calculate the value of $g_n$ on each element of $\P_n$ as in \eqref{estimator_partition}.  

In this paper, we present a novel family of algorithms by relaxing the condition that the rules constitute a partition of $\Rset^d$. We  consider collections of sets $\C_n = \C_n(\mathbf D_n)$ which asymptotically cover $\Rset^d$ (see \ref{H:collectempcoverage}). We call such a collection a quasi-covering, and a quasi-partition if the sets do not intersect, which we do not need to assume in general. Such a quasi-covering is turned into a quasi-partition $\P(\C_n)$ using the definition below. Then according to the ERM principle we define a quasi-covering estimator $g_n$ by \eqref{estimator_partition} where $\P_n=\P(\C_n)$. 	The main interest of quasi-covering estimators compared to partitioning estimators is its interpretability. This will be discussed in Section~\ref{sec:interpretability}.

\begin{definition}
	Let $\C$ be a finite collection of sets of $\Rset^d$ and let $\mathbf c=\bigcup\limits_{\r\in\C}\r$. We define the activation function as 
	$$ \varphi_\C: \Rset^d \mapsto 2 ^{\C};\qquad \varphi_\C(\x)=\{\r\in \C: \x\in \r\}. $$
	Then $\P(\C)$, the partition of $\c$ generated from $\C$, is defined as
	$$\P(\C):= \varphi_\C^{-1}(Im (\varphi_\C)).$$
	We  introduce the maximal (resp. minimal) redundancy of $\C$ as 
	$$M(\C):= \max_{\x\in\c}\#\varphi_\C(\x)\,,\qquad m(\C):= \min_{\x\in\c}\#\varphi_\C(\x).$$
	Moreover, the cells of $\P(\C)$ which are included in an element of $\C'\subset \C$ are gathered in $\P_\C(\C')$ : 
$$\P_\C(\C') := \{A\in\P(\C): \exists\r\in\C',\, A\subseteq\r\}.$$
\end{definition}
We provide a more explicit characterization of $\P(\C)$ in Proposition \ref{lemma:PC}. 
	\begin{proposition}\label{lemma:PC}
		Let $\C$ be a finite collection of sets of $\Rset^d$. Then 
		\begin{equation*}
			\P(\C) = \Bigl\{ \bigcap_{\r \in \tilde\C} \r \setminus \bigcup_{\r \in \C\setminus\tilde\C} \r : \tilde\C \subseteq \C \Bigr\}\setminus\{\emptyset\}.
		\end{equation*}
	\end{proposition}
The proof of Proposition \ref{lemma:PC} is deferred to the appendix. Figures \ref{fig:covering} and \ref{fig:partition} illustrate the fact that $\P(\C)$ can contain (much) more sets than $\C$ itself. 
	
\begin{figure}[ht]
	\centering
	\begin{minipage}{.48\textwidth}
		\centering
		\includegraphics[width=1\linewidth, height=0.25\textheight]{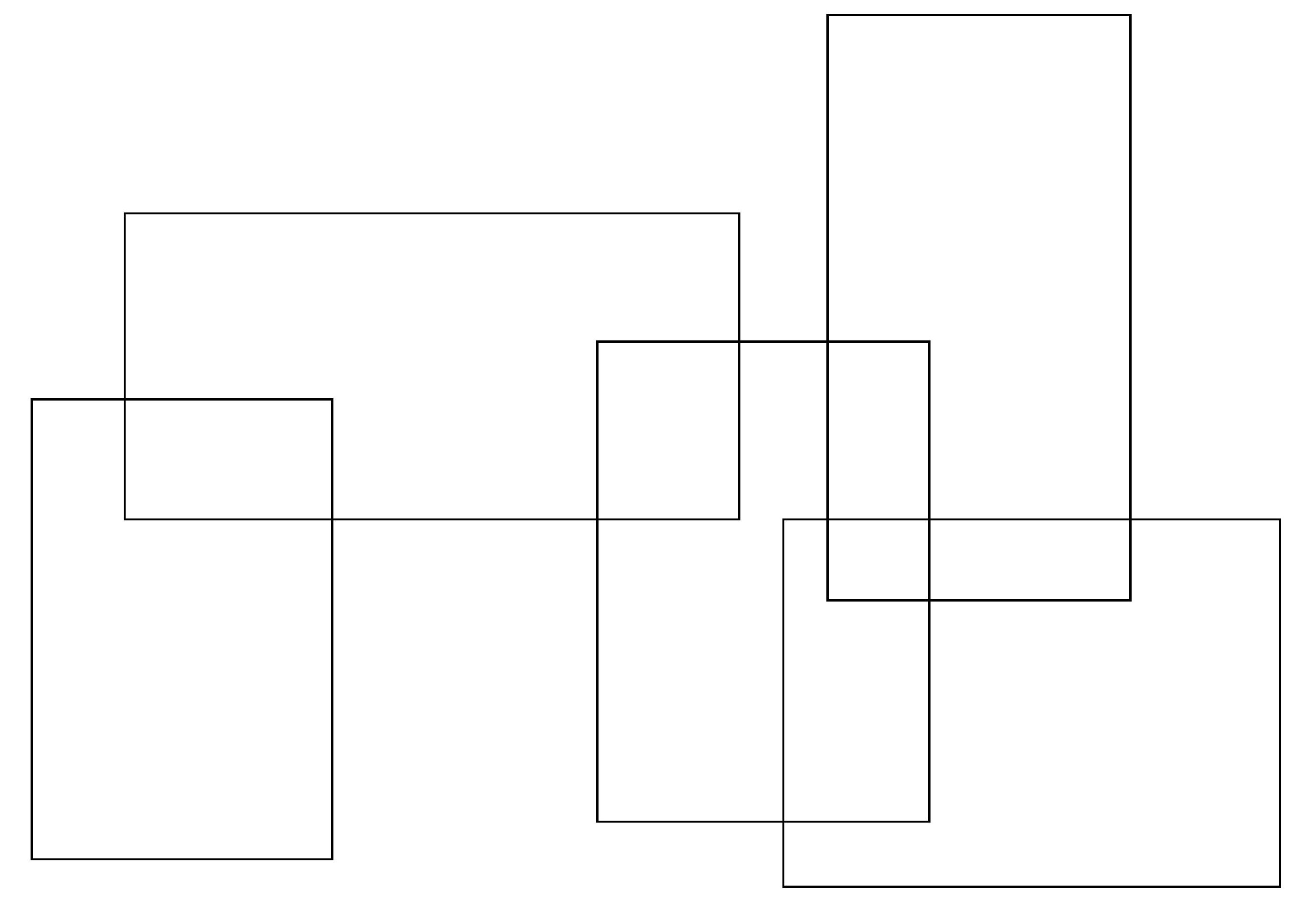}
		\caption{\label{fig:covering}The 5 overlapping elements of $\C$.}
		
	\end{minipage}%
	\begin{minipage}{0.48\textwidth}
		\centering
		\includegraphics[width=1\linewidth, height=0.25\textheight]{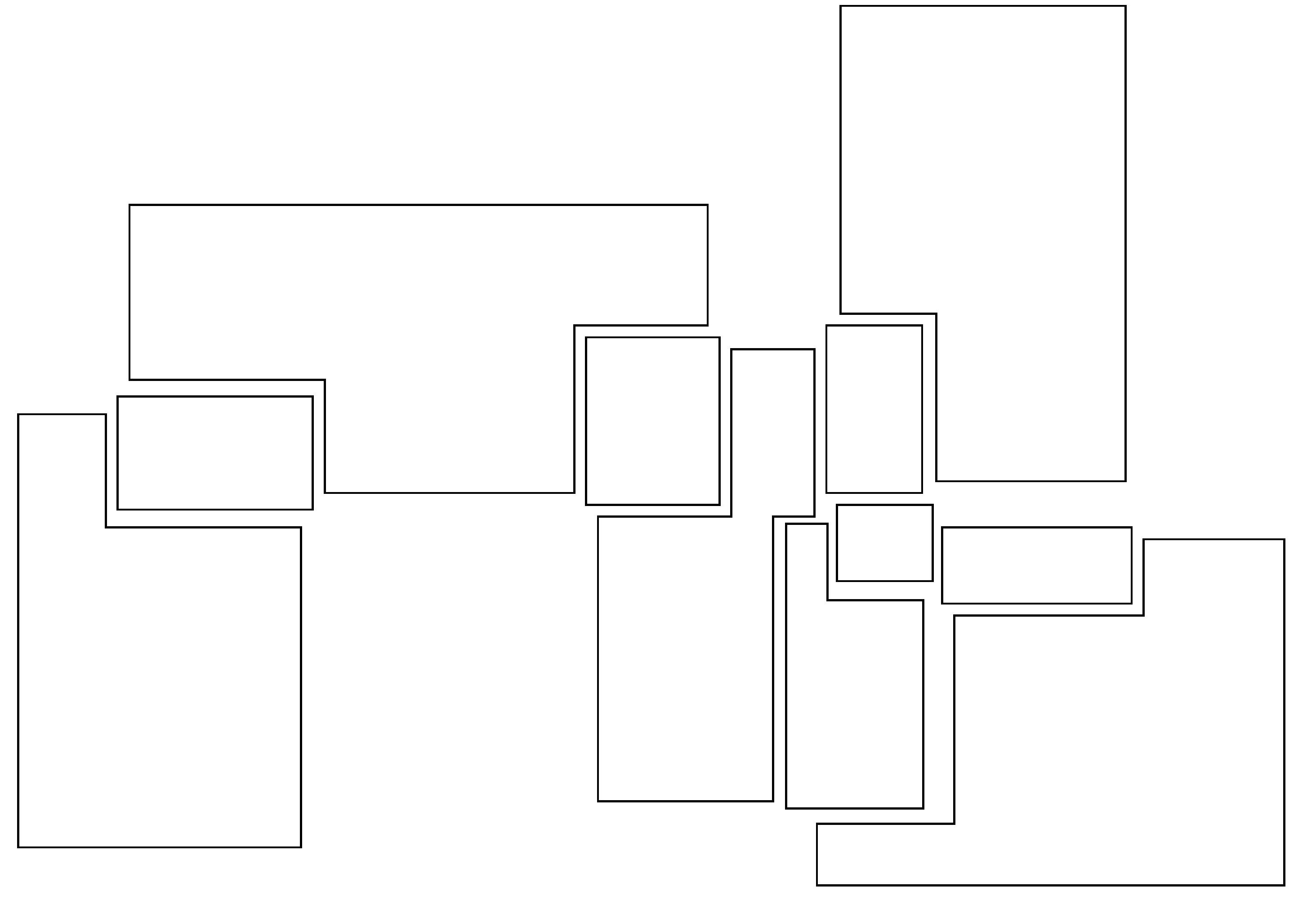}
		\caption{\label{fig:partition}The 11 cells of the partition $\P(\C)$.}
	\end{minipage}
\end{figure}
		
	As the construction of $\P(\C)$ can be time consuming, it is noteworthy that there is no need to fully describe $\P(\C)$ to compute the value of $g_n(\x)$ for some $\x\in\Rset^d$. The trick is to identify the unique cell of $\P(\C)$ which contains $\x$. By creating binary vectors of size $\#\C$, whose value is $1$ if $\x$ fulfills the rule's condition and $0$ otherwise, this cell identification is a simple sequence of vectorial operations. Figure~\ref{fig:partition_trick} is an illustration of this process (we refer to \cite{Margot18} for more details).

	As compared to ours, the predictors designed by Random Forests (RF) and other tree ensembles algorithms do not satisfy the ERM principle \eqref{argmin}: they are based on the averaging of the predictions of the activated rules at the point $\x$. This averaging is justified empirically and theoretically as it lowers the variance of the prediction thanks to the independence of the random rules as explained in \cite{Breiman01}. The averaging does not decrease the approximation error of the prediction and there are few results about their consistency, we may cite \cite{Denil13,Scornet15} for RF only. On the opposite, our estimator defined by \eqref{estimator_partition} is the empirical conditional expectation on the cell of $\P(\C_n)$. Under some conditions, the approximation error of the prediction is low thanks to the ERM principle. The variance will be controlled through the significance condition (see Definition~\ref{def:suitable_covering} below) as in the RIPE algorithm of \cite{Margot18}. 


\begin{figure}[ht]
	\centering
	\makebox{\includegraphics[scale=0.35]{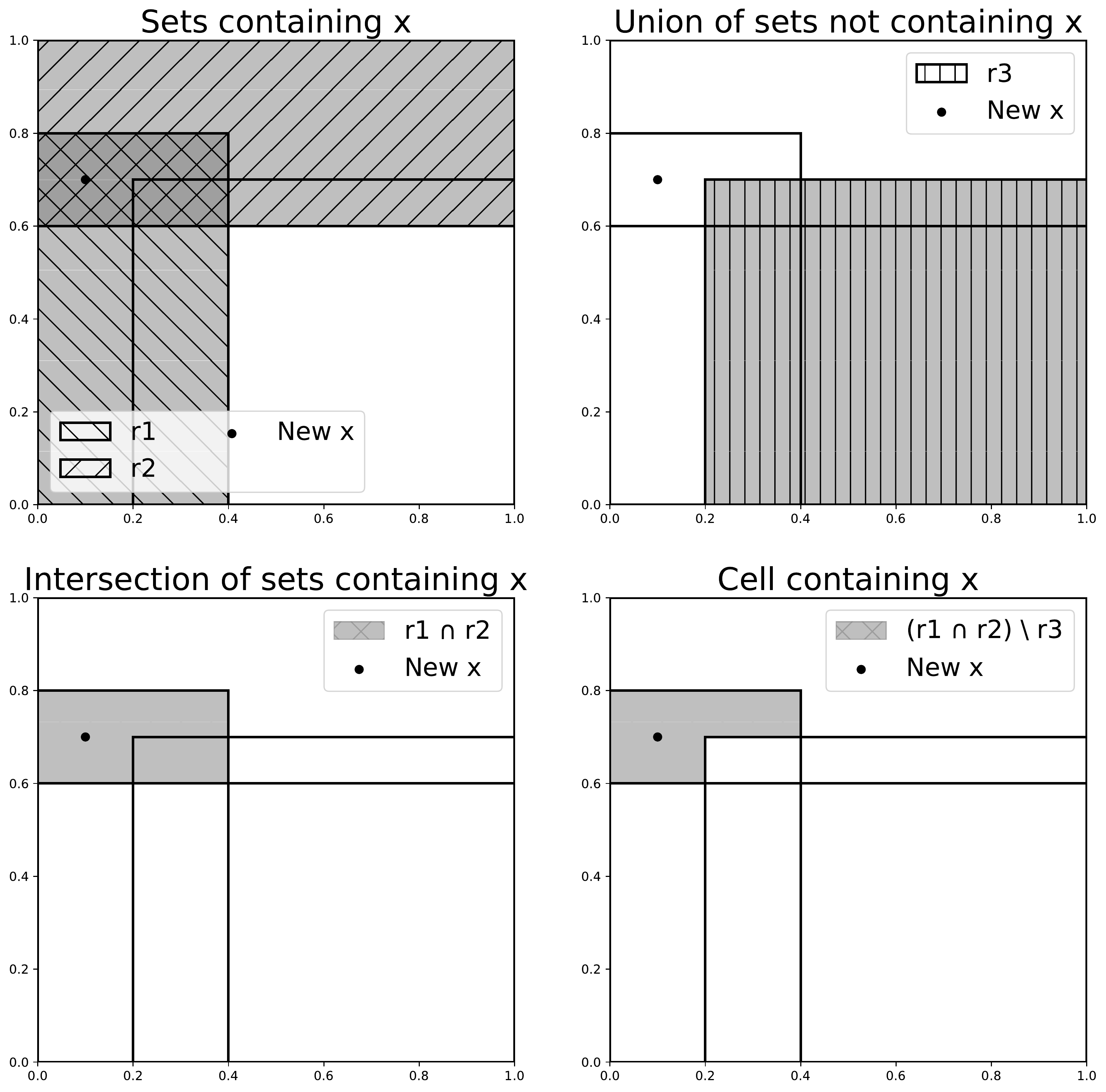}}
	\caption{\label{fig:partition_trick}Evaluation steps of the cell containing $\x = (0.1, 0.7)$ of the partition generated from the covering of $ [0,1]^2$, $\C =\{ \r_1, \r_2, \r_3\}$. Using partition from a covering allows to generate complex cells with a simple interpretation $(\r_1 \ And \ \r_2)$, where a classical partitioning algorithm cannot. Note that the condition $\x$ satisfies $(\r_1 \ And \ \r_2)$ implicitly implies that $\x$ does not satisfy $\r_3$.}
\end{figure}

\subsection{Interpretability}\label{sec:interpretability}

In many sensitive areas, such as health care, justice, defense or asset management the importance of interpretability in the decision-making process has been underlined. As explained in \cite{Lipton18}, there are different meanings of \emph{interpretability} depending on the desiderata of the user and the expected properties of the algorithms. In this paper, we embrace the intuitive definition of model interpretability of \cite{Biran17}: \textbf{Interpretability is the degree to which an observer can understand the cause of a decision.}  \emph{Interpretable} models should provide a parsimonious characterization of an estimator of $g^*$. Nowadays, the most popular and accurate algorithms for regression, such as Support Vector Machines, Neural networks, Random Forests,\ldots are no interpretable. This lack of interpretability comes from the complexity of the models they generate. We call them black box models. We state that the novel family of (quasi-)\emph{covering algorithms} described in Section \ref{sec:Application} can achieve different interpretability-accuracy trade-off by reducing the complexity of the generated models while being accurate enough, maintaining weak consistency guarantees. 

\begin{figure}[ht]
	\centering
	\includegraphics[scale=0.9]{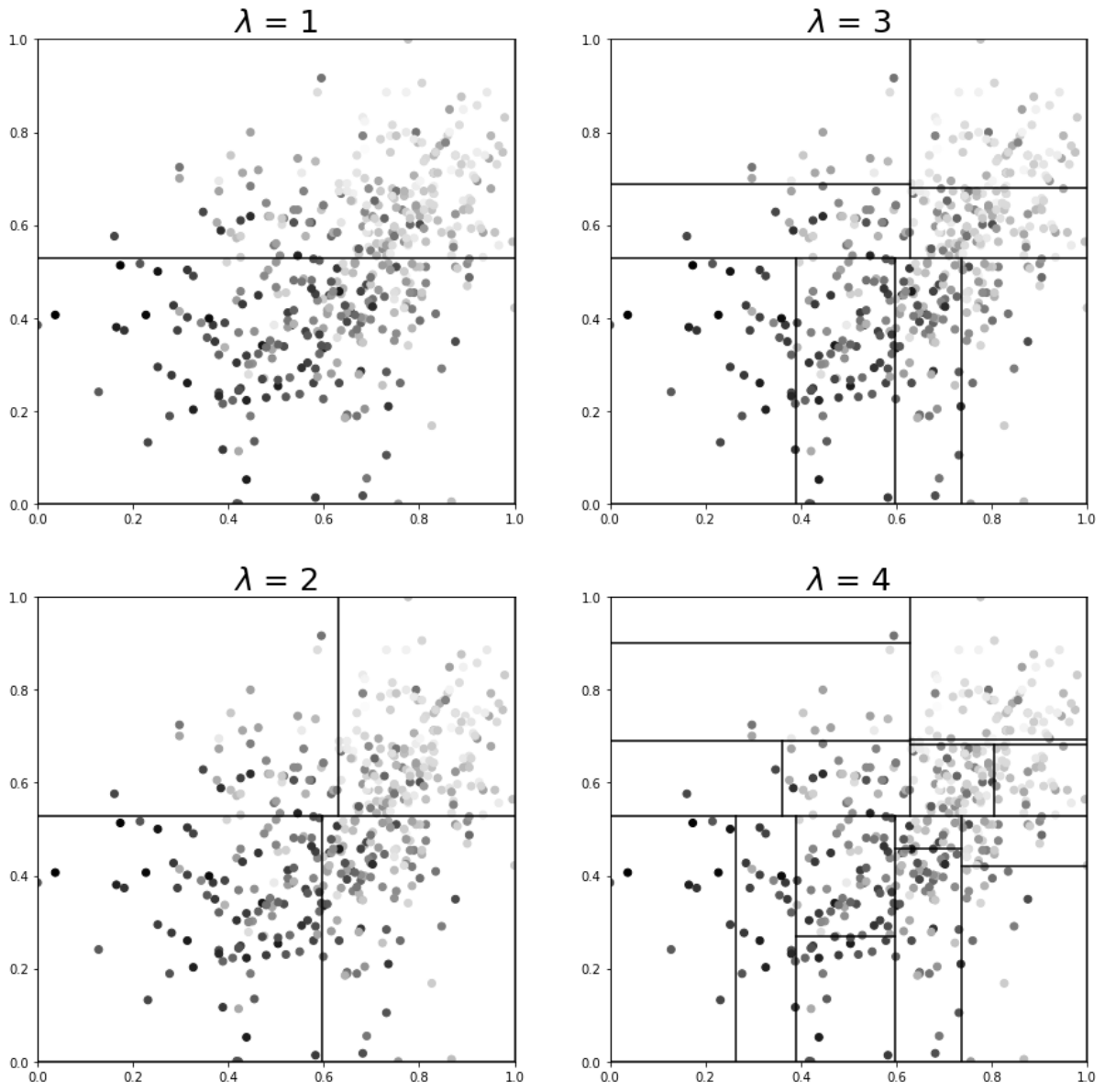}
	\caption{Partitions generated by fully deployed decision tree algorithm  without pruning for a maximal depth $\lambda \in \{1,2,3,4\}$. Here the number of rules satisfies $2^\lambda$,  $\lambda \in \{1,2,3,4\}$.}
	\label{fig:partition_tree}
\end{figure}

We distinguish two main approaches to generate interpretable prediction models. The first approach is to create black-box models and then to summarize them in a so-called \emph{post-hoc} interpretable model. Recent researches propose to use explanation models such as \emph{LIME} \citep{Ribeiro16}, \emph{DeepLIFT} \citep{Shrikumar19} or \emph{SHAP} \citep{Lundberg17} to interpret black-box models. These explanation models  measure the importance of a feature on the prediction process (see \cite{Guidotti18} for a survey of existing methods). The other possibility is to use \emph{intrinsic} interpretable algorithms  that directly generate interpretable models.
 
The interpretability of a rule-based estimator is achieved when the length $k$ of each rule and the number of rules are small, i.e. when the corresponding collection of rules is simple. Considering that a collection $\C$  with one rule of length $k$ is just as interpretable as a collection with  $k$ rules of length $1$ since they consists in the same number of tests, interpretability gets a certain additivity property. With this in mind, we are able to quantify the interpretability index of a collection $\C$.
\begin{definition}
The \emph{interpretability index} of a collection of rules $\C$ is defined by
\begin{equation}\label{eq:inter}
	Int(\C) := \sum_{\r \in \C} \text{length}(\r),
\end{equation}
\end{definition}
By convention, we define the interpretability index of the estimator $g_n$ generated by a set of rules $\C_n$ as $Int(g_n)=Int(\C_n)$. \begin{remark}
The interpretability index $Int(g_n)$  does not take into account the notion of interaction described in \cite{Friedman08}. Indeed, a low interpretability index does not inform whether $g_n$ exhibits interactions between variables or not. Many interactions may exhibit from a set of $k$ rules of length $1$ whereas it is less the case for a unique rule of length $k$. Due to the additivity property, both cases have the same interpretability index.
\end{remark}

We aim at designing estimators which combine consistency guarantees keeping an interpretability index as small as possible.
To demonstrate the consistency of a partition-based estimator $g_n$, results such as Theorem 13.1 in \cite{DistributionFree}, which is a generalization of Stone's Theorem to data-dependent partitioning estimates, are usually applied. Then, ``it is required that the diameter of the cells of the data-dependent partition [...] converge in some sense to zero" (Condition~(13.10) in \citealp{DistributionFree}). However, as illustrated in Figure~\ref{fig:partition_tree}, such approaches based on fine partitions have a high interpretability index. A cell shrinkage condition such as Condition~(13.10) in \cite{DistributionFree} necessarily implies  that the number of elements in the partition $\P_n$ tends to infinity and that the interpretability index also tends to infinity.

For an estimator defined on a data-dependent quasi-covering, we can get rid of this condition and avoid the condition of shrinkage of the sets in the collection $\C_n$. Any prediction can potentially be explained by a small set of  rules that are interpretable: See Table~\ref{tab:rules_ozone}  in Section \ref{sec:Application} for an example. The generated partition may be finer and more complex than a classical data-dependent partition, while the explanation of the quasi-covering is easily understood by humans.  

Despite the fact that we have not yet proven the parsimony of the selected set of rules in general, the use of quasi-covering instead of partition and conditions other than the usual ones drastically reduce the number of rules in simple cases. This is discussed in detail in Section~\ref{sec:discussionInterpretability}.


We prove the consistency of a quasi-covering based estimator $g_n$ by carefully designing the covering elements. The key concept is that of suitable data-dependent quasi-covering introduced in Section~\ref{sec:mainresult}. Proposition~\ref{prop:emp} provides rates of convergence of empirical conditional expectations. This preliminary result of independent interest is crucial  to prove the weak consistency of suitable data-dependent covering estimators stated in Theorem~\ref{th:consistance}.  Our main result Theorem~\ref{th:consistance} is proven in Section \ref{sec:proof} while the proofs of intermediate results are postponed to the appendix.
We apply our quasi-covering approach on artificial data and real data using different rule generators, and we compare it to usual existing  algorithms in Section~\ref{sec:Application}.

\section{Consistent prediction based on suitable quasi-coverings}\label{sec:mainresult}

	In order to ensure the consistency of quasi-covering based algorithms we have to introduce the notion of \emph{suitable} sequence of data-dependent quasi-coverings. 
	
	\subsection{Suitable data-dependent quasi-coverings}
	
	\subsubsection{Definition}
	We use the classical notation $x_+=\max\{x,0\}$ for any $x\in\Rset$.
	\begin{definition}\label{def:suitable_covering}
		Let $\alpha\in [0, 1/2)$. We call a sequence $(\C_n)_{n \ge 1}$ of data-dependent finite collections of sets $\r$ of $\Rset^d$ \emph{$\alpha$-suitable} if it satisfies:
		\begin{enumerate}
			\item 
			\begin{hypo}\label{H:subsetcoverage}
				The set coverage condition
			\end{hypo}
			\begin{equation*}
				\Q_n(\r) > \,n^{-\alpha}, \qquad \r\in\C_n, n \geq 1.
			\end{equation*}
			\item
			\begin{hypo}\label{H:collectempcoverage}
				The collection coverage condition
			\end{hypo}
			\begin{equation*}
				\Q_n(\c_n^c) = o_\Pr(1)
			\end{equation*}
			where $\c_n = \cup_{\r \in \C_n} \r$.
			\item
			\begin{hypo}\label{H:significance}
				The significance condition
			\end{hypo}
			
			For some consistent estimator $(\sigma_n^2)$ of $\sigma^2$ and sequences $\beta_n = o_\Pr(1)$ and $\varepsilon_n = o_\Pr(1)$, it holds for all $n\geq 1$
			\begin{equation*}
				\C_n=\C_n^s\cup\C_n^i
			\end{equation*}
			where the significant sets $\C_n^s$ are defined by
			\begin{multline}\label{eq:Cns_emp}
				\C_n^s := \Big\{\r\in\C_n : \beta_n\big| \E_n[Y | \X \in \r] - \E_n[Y]\big| \geq \sqrt{(\V_n(Y | \X \in \r ) - \sigma_n^2)_+}\Big\}
			\end{multline}
			and the insignificant sets $\C_n^i$ are defined by
			\begin{equation}\label{eq:Cni_emp}
				\C_n^i := \left\{\r\in\C_n\setminus \C_n^s : \varepsilon_n \geq \sqrt{ ( \V_n(Y | \X \in \r ) - \sigma_n^2)_+}\right\}.
			\end{equation}
			
			\item
			\begin{hypo}\label{H:redundancy}
				The redundancy condition
			\end{hypo}
			
			The redundancies of $\C_n^s$ and $\C_n^i$ satisfy
			\begin{equation}\label{eq:redundancy_significant}
				\frac{M(\C_n^s)}{m(\C_n^s)}= o_\Pr(\beta_n^{-2}\wedge n^{1/2 - \alpha})
			\end{equation}
			and
			\begin{equation}\label{eq:redundancy_insignificant} 
				\frac{M(\C_n^i)}{m(\C_n^i)} = o_\Pr(\varepsilon_n^{-2}\wedge n^{1/2 - \alpha}).
			\end{equation}
		\end{enumerate}
	\end{definition}


	 	A few remarks on the conditions of Definition~\ref{def:suitable_covering} are in order. The set coverage condition \ref{H:subsetcoverage} controls the minimal number of observations in a set of the collection $\C_n$, ensuring that the empirical conditional expectation is close to the conditional expectation given any of the sets in $\C_n$. The collection coverage condition \ref{H:collectempcoverage} guarantees that $\C_n$ covers a sufficient part of $\Rset^d$. The significance conditions \ref{H:significance} control the variance of the rules in two different ways; The significant rules are predicting significantly differently than the average and may have a large variance, the insignificant rules are not significant and have a small variance. The redundancy condition \ref{H:redundancy} controls how far from a partition the collections $\C_n^s$ and $\C_n^i$ can be. This condition is fulfilled if the sets in these collections are mutually disjoint.
		
\begin{remark}
	The significant condition \eqref{eq:Cns_emp} is not related to a condition on the diameter of the set $\r$. It is intended to detect areas of the feature space where the conditional expectation of $Y$ is noticeably different from its unconditional expectation over the whole feature space. See Example~\ref{exampleSignifLargeDiameter}. The insignificant condition \eqref{eq:Cni_emp} can be proved to hold under a shrinkage condition on the  diameter of the set, see Proposition \ref{prop:diam}. This is not a necessary condition. 
\end{remark}

\begin{remark}\label{rem:sigmaknown}
	We may assume in the following that the noise variance $\sigma^2$ is known or that a good enough estimator $\sigma_n^2$ exists. Theorem~\ref{th:consistance} proves the consistency of the procedure under the assumption that $|\sigma^2_n-\sigma^2|=O_\Pr(n^{\alpha-1/2})$. One refers to \cite{Liitiainen09} as a reference on the difficult question of the bias-variance trade-off of some residual variance estimators. Most of the existing estimators are based on Nearest-Neighbor methods. Random forests were not used for residual variance estimation until recently, see \cite{ramosaj19}.  None of these results is satisfactory in our setting of rule-based procedure. We let for a future work the study of a good enough estimator $\sigma_n^2$ based on a covering algorithm. We will use on real data in Section \ref{sec:Application} the smallest of the empirical variances of the rules. 
\end{remark}

	\begin{remark}\label{rem:redundancy}
	An easy way to ensure \ref{H:redundancy} is to constrain the proportion of any element of $\C_n$ which can be included in the union of the others to be small. Let $(\C_n)$ be a sequence of collections of sets $\r$ of $\Rset^d$ that fulfills~\ref{H:subsetcoverage} and $\gamma \in (0,1)$. We consider $(\r_i)_{1\le i\le \# \C_n}$ any ordering of $\C_n$. If 
	$$\Q_n\Big(\r_{i+1} \bigcap\Big\{ \bigcup\limits_{1\le j\le i} \r_j\Big\}\Big) \leq \gamma\, \Q_n(\r_{i+1})\,,\qquad 1\le i\le \# \C_n-1\,,$$
	then the cardinality of $\C_n$ is upper bounded by $\frac{n^\alpha }{1 - \gamma}$ for $n$ sufficiently large. Indeed, by the inclusion-exclusion principle we get
	$$ 1 \geq \Q_n(\cup_{1\leq i\leq \#\C_n} \r_i)=\sum_{i=1}^{ \# \C_n}\Q_n(\r_i \setminus \cup_{1\le j\le i-1} \r_j)\ge \# \C_n\, (1-\gamma)\, n^{-\alpha}.$$
	Thus \ref{H:redundancy} can be checked for any $\alpha \in [0, 1/4)$, using the fact that $M(\C_n^s)$ and $M(\C_n^i)$ are smaller than $\frac{n^\alpha }{1 - \gamma}$ and setting $\beta_n = o_\Pr(n^{  - \alpha/2})$ and $\varepsilon_n = o_\Pr(n^{ - \alpha/2})$.
\end{remark}

\subsubsection{Advantages of suitable quasi-coverings}\label{sec:discussionInterpretability}

We discuss how our approach can provide estimators with low interpretability indices by considering quasi-coverings instead of partitions on the one hand and by introducing  conditions on the elements of the quasi-coverings which are completely different from Condition~(13.10) in \citet{DistributionFree} on the partition cells on the other hand.

We already discussed and illustrated in Figures~\ref{fig:covering} and \ref{fig:partition} that a quasi-covering $\C_n$ contains less elements than the quasi-partition $\P(\C_n)$. Then its interpretability index is lower. 

Moreover, $\P(\C)$ contain more complex elements than those of a partition-based algorithm: the cells in $\P(\C)$ in Figure~\ref{fig:partition} are not necessarily conjunctions of simple tests on coordinates as the sets in Figure~\ref{fig:covering}. Our covering-based approach generates complex cells without sacrificing interpretability since the algorithm keeps small the number of elements of $\C$ which have to be designed. 

Let us illustrate this point in the following example considering rules designed as hyperrectangles which is usual, see Remark \ref{rem:hyperrect}.
\begin{proposition}\label{prop:compPartitionCovering}
	Let $C = [0,1]^d \subset \Rset^d$. 
	\begin{enumerate}
		\item The minimal cardinality of a partition $\P$ of $\Rset^d$ by hyperrectangles such that $C\in\P$ is $2d+1$.
		\item The minimal cardinality of a covering $\C$ of $\Rset^d$ by hyperrectangles such that $C\in\P(\C)$ is $2$.
	\end{enumerate}
\end{proposition}
The proof of Proposition \ref{prop:compPartitionCovering} is deferred to the appendix and is a consequence of the fact that $\Rset^d\setminus [0,1]^d$ is a more complex set than an hyperrectangle.

On the other hand, our Condition~\ref{H:significance} suits better the requirement of interpretability than Condition~(13.10) of \cite{DistributionFree} as they do not imply any shrinkage of the sets in the collection $\C_n$ and the number of sets in $\C_n$ does not necessarily have to grow with $n$. We illustrate in the following example this point when $g^*$ itself is rule-based with a low interpretability index: in this important situation, an estimator with  low interpretability index is expected.

\begin{example}\label{exampleSignifLargeDiameter}
	The condition involved in \eqref{eq:Cns_emp} can hold for a set $\r$ with arbitrary diameter that does not satisfy Condition~(13.10) of \cite{DistributionFree}. Consider the case $g^*=\ind{A}$ for some Borel set $A$ such that $0<\Pr(\X\in A)<1$ and assume that $\sigma^2$ is known.  Then $\r=A$ is a significant set as it satisfies the condition involved in \eqref{eq:Cns_emp} for some well-chosen $\beta_n = o_\Pr(1)$. Indeed, from the Strong Law of Large Numbers $k_n:=\#\{i:\X_i\in A\}\sim n\Pr(\X\in A)$ a.s. as $n\to \infty$. Then repeated application of the Central Limit Theorem yields
	\begin{align*}
		\big| \E_n[Y | \X \in A] - \E_n[Y] \big| &\geq \E_n[Y | \X \in A] - \E_n[Y] \\
		&= 1 - \dfrac{k_n}n + \dfrac1{k_n}\sum_{i=1}^{ n}Z_i\ind{\X_i\in A} - \dfrac1{n}\sum_{i=1}^{ n}Z_i \\
		&= 1 - \Pr(\X\in A) + O_\Pr(n^{-1/2})
	\end{align*}
	and
	\begin{align*}
		(\V_n(Y | \X \in A) - \sigma^2)_+ 	&\leq | \V_n(Y | \X \in A) - \sigma^2| \\	
		&= \bigg|\dfrac1{k_n}\sum_{i=1}^{ n}Z_i^2\ind{\X_i\in A}-\Big(\dfrac1{k_n}\sum_{i=1}^{ n}Z_i \ind{\X_i\in A}\Big)^2 - \sigma^2\bigg|\\
		&= O_\Pr(n^{-1/2}).
	\end{align*}
	It follows that \eqref{eq:Cns_emp} holds for $\r = A$ with 
	\begin{equation*}
		\beta_n = \frac{\sqrt{(\V_n(Y | \X \in A) - \sigma^2)_+}}{\big| \E_n[Y | \X \in A] - \E_n[Y] \big| } = O_\Pr(n^{-\frac 14}).
	\end{equation*}
	For similar reasons \eqref{eq:Cns_emp} also holds for $\r = A^c$ with 
	\begin{equation*}
		\beta_n' = \beta_n \vee \frac{\sqrt{(\V_n(Y | \X \in A^c) - \sigma^2)_+}}{\big| \E_n[Y | \X \in A^c] - \E_n[Y] \big| } = O_\Pr(n^{-\frac 14}).
	\end{equation*}
	Finally it can easily be checked that $\C_n = \{\r\in\{A,A^c\} : \Q_n(\r) > n^{-\alpha}\}$ defines an $\alpha$-suitable sequence of quasi-partitions for any $\alpha\in(0,\frac 12)$. The number of rules in $\C_n$ tends to $2$.
\end{example}
	 	 
	 \begin{remark}\label{rem:abc} In Section~\ref{sec:Application} and for many usual procedures, only hyperrectangles are considered as rules. In Example~\ref{exampleSignifLargeDiameter}, if $A$ is assumed to be an hyperrectangle itself, $A^c$ can be written as the union of a set of hyperrectangles as in the proof of Proposition~\ref{prop:compPartitionCovering} and similar arguments as above apply to prove that an $\alpha$-suitable sequence of quasi-coverings which are all subsets of a small set of hyperrectangles can thus be defined accordingly. 
	 \end{remark}
	 
	Note that in Proposition~\ref{prop:compPartitionCovering}, Example~\ref{exampleSignifLargeDiameter}, and Remark~\ref{rem:abc}, the number of elements of the quasi-coverings and the interpretability index are bounded as $d$ or $n$ grows whereas they would need to tend to infinity for any method partition-based or relying on Condition~(13.10) of \cite{DistributionFree}.

Besides, another advantage is that the check of the conditions involved in \eqref{eq:Cns_emp} and \eqref{eq:Cni_emp} for elements of a quasi-covering can be done simultaneously with parallel computing on the contrary to any condition on partition cells that depend on each other.

\subsection{Consistency of data-dependent quasi-covering algorithms}

\subsubsection{Generalization of the partitioning number}\label{sec:Partitioning}	

Conditions are required to control the complexity of the family of partitions $\P(\C_n)$ that the algorithm can generate. These conditions use some concepts introduced in \cite[Sec. 1.2]{Nobel96} (see also \cite[Def 13.1]{DistributionFree}). The standard definitions have to be adapted since we consider collections which do not necessarily cover $\Rset^d$. To discriminate with the standard definitions for partitions we denote these quantities with tildes.
	\begin{definition}\label{def:nbr_partition} Let $\Pi$ be a family of finite collections of disjoint sets of $\Rset^d$. 
		\begin{enumerate}
			\item The maximal number of sets in an element of $\Pi$ is denoted by
			\begin{equation*}
				\widetilde\M(\Pi):= \sup \left\{ \#\P: \P \in \Pi \right\}.
			\end{equation*}
			\item For a set $\x_1^n = \{\x_1, \dots, \x_n\} \in (\Rset^d)^n$, let 
			\begin{equation*}
				\widetilde\Delta(\x_1^n, \Pi):= \#\big\{ \{\x_1^n \cap A: A \in \P \}\setminus\{\emptyset\} : \P \in \Pi \big\}.
			\end{equation*} 
			\item The partitioning number $\widetilde\Delta_n(\Pi)$ of $\Pi$ is defined by			
			\begin{equation*}
				\widetilde\Delta_n(\Pi):= \max \limits_{\x_1^n \in (\Rset^d)^n} \widetilde\Delta(\x_1^n, \Pi).
			\end{equation*}
		\end{enumerate}		
	\end{definition}
	The partitioning number is the maximal number of different collections of disjoint non-empty subsets of any $n$-point set that can be induced by elements of $\Pi$ and is smaller than or equal to $(n+1)^n$, which is an upper bound of the number of collections of disjoint subsets of $\x_1^n$ as the number of maps from $\{1,\dots,n\}$ to $\{0,\dots,n\}$.

	Classical consistency theorems require to control the partitioning number of the family of partitions $\Pi_n$ over which the estimators are defined (see (13.7) and (13.8) in \citealp[Theorem 13.1]{DistributionFree}). As illustrated in Figure~\ref{fig:partition}, the complexity of the sets in $\P(\C_n)$ may be high even if the complexity of the sets in $\C_n$ is low. This fact makes the partitioning number of the family of partitions tough, if not impossible, to bound because no control on the shape of the cells is available. 
	
	Usually $\Pi_n$ is chosen as $\{\P(\C_n(\d_n)) : \d_n\in\S^n\}$. It can be useful to take into account the building process of the family of sets $\C_n(\d_n)$ in the evaluation of the partitioning number when the partitioning number of $\{\P(\C_n(\d_n)) : \d_n\in\S^n\}$ cannot be controlled easily. This is what the definition below enables by considering $\Pi_n$ as a refinement of $\{\P(\C_n(\d_n)) : \d_n\in\S^n\}$.
	
	\begin{definition}\label{def:refiement}
		Let $\mathcal F$ and $\mathcal F'$ be two families of finite collections of disjoint sets. $\mathcal F'$ is called a refinement of $\mathcal F$ if 
	\begin{equation*}
		\forall F\in\mathcal F, \quad\exists F'\in\mathcal F', \quad\forall A\in F,\quad A = \bigcup_{A'\in F' : A'\subseteq A}  A'.
	\end{equation*}
	\end{definition}

\subsubsection{Main result}

Our main result provides the consistency of estimators based on suitable quasi-covering sequences. 

	\begin{theorem}\label{th:consistance}
		Assume that $g^*$ and $\Q$ satisfy \ref{H:indep}, \ref{H:bounded}. Let $\alpha\in[0,\frac 12)$ and $(\C_n)$ be an $\alpha$-suitable sequence of finite collections of sets of $\Rset^d$ such that:
				
		\begin{hypo}\label{H:collectthcoverage}
			\begin{equation*}
				\Q(\c_n^c) = O_\Pr(n^{\alpha-1/2})
			\end{equation*}
			where $\c_n = \cup_{\r\in\C_n} \r$;
		\end{hypo}
		
		\begin{hypo}\label{H:partitions}
			\begin{equation*}
				\widetilde\M(\Pi_n) = o(n) \qquad\text{and}\qquad \log(\widetilde\Delta_n(\Pi_n) ) = o(n)
			\end{equation*}
			where $\Pi_n$ is a refinement of $\{\P(\C_n(\d_n)) : \d_n\in\S^n\}$;
		\end{hypo}
		
		\begin{hypo}\label{H:Donsker}
			\begin{equation*}
				\forall n\in\Nset^*, \forall \d_n\in\S^n, \{\r \times\Rset, \r \in \C_n(\d_n)\} \subseteq \Br
			\end{equation*}
			where $\Br$ is a $\Q$-Donsker class.
		\end{hypo}
		Assume moreover that the estimator $(\sigma_n^2)$ involved in Definition~\ref{def:suitable_covering} is such that 
			\begin{equation*}
				|\sigma^2_n-\sigma^2|=O_\Pr(n^{\alpha-1/2}).
			\end{equation*}
		Then 
		\begin{equation*}
			g_n = \mathop{\sum_{A\in\P(\C_n)}}_{\Q_n(A)>0} \E_n[Y \mid \X \in A] \ind{A}
		\end{equation*} 
		is weakly consistent:
		\begin{equation*}
			\ell \left( g^*, g_n\right) = o_\Pr(1).
		\end{equation*}
	\end{theorem}

\begin{remark}\label{rem:hyperrect}
Many algorithms as CART generate rules as in \eqref{rule_form} where each simple condition $(\X\in c)$ consists in some distinct coordinate of $\X$ belonging to an interval of $\Rset$. In Section \ref{sec:Application} we use such rule generators for which \ref{H:Donsker} is automatically satisfied. Indeed let $\mathcal{H}$ be the set of all hyperrectangles of $\Rset^d$:
\begin{equation*}
\mathcal{H} := \Big\{I_1\times\ldots\times I_d\subset \Rset^d : I_i \text{ is an interval of } \Rset \Big\} \cup \{\emptyset \}.
\end{equation*}
Any set of rules $\C_n$ such that $\C_n \subseteq \mathcal{H}$, which is always the case of the sets of rules considered in Section \ref{sec:Application}, fulfills  \ref{H:Donsker} since it is a VC class (see \citealp{wenocur1981}) and VC classes are $\Q$-Donsker \citep[see for example][Lemma 19.15 and comments]{VanDerVaart00}.
\end{remark}

\section{Proof of Theorem~\ref{th:consistance}}\label{sec:proof}
In order to prove the main theorem, we need some preliminary results based on the notions of $\Q$-Donsker class and outer probability. The outer probability, defined for $A \subseteq \Omega$ by $\Pr^*(A):= \inf \left\{ \Pr(\tilde A): A \subset \tilde A, \tilde A \in \A \right\}$, is introduced to handle functions which are not necessarily measurable. The usual notion of boundedness in probability for sequences of random variables is generalized because sequences of maps are considered  with values in metric spaces which are not Euclidean spaces (thus bounded and closed sets need not be compact) and which are not guaranteed to be measurable. See \cite[Chapter~18]{VanDerVaart00} for details. 

\begin{definition}\cite[Chapter 18]{VanDerVaart00}\label{def:as-tight} 
	A sequence $(M_n)_{n\in\Nset}$ of maps defined on $\Omega$ and with values in a metric space $(\D, d)$ is said to be asymptotically tight if 
	\begin{equation*}
		\forall \varepsilon > 0, \exists K\subset \D \text{ compact}, \forall \delta > 0, \limsup_{n \rightarrow \infty} \Pr^*(M_n\notin K^\delta) < \varepsilon, 
	\end{equation*}
	with $K^\delta = \{y\in\D : d(y,K)<\delta\}$. 
\end{definition}

\begin{remark}\label{rem:as-tight-R}
	If $\D = \mathbb R$, $(M_n)$ is asymptotically tight if and only if 
	\begin{equation*}
		\forall \varepsilon>0, \exists M>0 \text{ such that }\limsup_{n\rightarrow \infty} \Pr^*(|M_n| > M) < \varepsilon.
	\end{equation*}
\end{remark}

The notation $O_{\Pr^*}(1)$ stands for \textit{asymptotically tight} instead of the usual $O_\Pr(1)$ (\textit{bounded in probability}). 

For $f : \S \rightarrow \mathbb R$ in $\L ^1(\Q)$ we define $v_nf := \sqrt n (\int f d\Q_n - \int f d\Q)$ and consider the empirical process indexed by a set $\F$ of such functions: $\{v_n f : f\in\F\}$. 

\begin{definition}\cite[Section 19.2]{VanDerVaart00}\label{def:Donsker-class} 
	A class of functions $\F$ is called $\Q$-Donsker if the sequence of processes $\{ v_n f: f \in \F \}$ converges in distribution to a tight limit process in the space $\ell^\infty(\F)$.
\end{definition}
The limit process is then a $\Q$-Brownian bridge.

\begin{definition}
	A class of sets $\Br \subseteq \Br_\S$ is called $\Q$-Donsker if $\I_\Br := \{ \ind{A} : A \in \Br \}$ is a $\Q$-Donsker class of functions.
\end{definition}

If $\F$ is a $\Q$-Donsker class of functions, then the empirical process $((v_n f)_{f\in\F})_{n\in\Nset}$ is asymptotically tight as a sequence of maps with values in $\ell^\infty(\F)$ (this is a consequence of Prohorov's Theorem adapted to this framework~--~see \citealp[Theorem~18.12]{VanDerVaart00}). Keeping in mind that a compact set in $\ell^\infty(\F)$ is bounded, we have:
\begin{proposition}\label{prop:donsker}
	Let $\F$ be a $\Q$-Donsker class of functions. Then
	$$ \left\| \Q_n - \Q \right\|_{\F} = O_{\Pr^*}(n^{-1/2}),$$
	where for any $v : \F \rightarrow \Rset$, $\left\| v \right\|_{\F} = \sup_{f \in \F} \left| v(f) \right|$.
\end{proposition}

\begin{remark}
	If $\Br \subseteq \Br_\S$ is a $\Q$-Donsker class of sets, where $\Br_\S$ is the Borel set on $\S$, then 
	$$ \left\| \Q_n - \Q \right\|_\Br = O_{\Pr^*}(n^{-1/2}),$$
	where for any $v : \Br \rightarrow \Rset$, $\left\| v \right\|_\Br = \sup_{A \in \Br} \left| v(A) \right|$.
\end{remark}

\begin{remark}\label{rem:P}
	It can be checked that if $(Z_n)_{n\in \Nset}$ is a sequence of non-negative random variables, $(a_n)_{n \in \Nset} \in (\Rset^+)^\Nset$ such that $a_n = o_\Pr(1)$ and $(M_n)_{n\in \Nset}$ is a sequence of maps (non necessarily measurable) such that $M_n = O_{\Pr^*}(1)$ and $Z_n \leq a_nM_n$ for any $n$, then $Z_n \underset{n \to +\infty}{\overset{\Pr}{\longrightarrow}} 0$.
\end{remark}

\subsection{Empirical estimation of conditional expectations}\label{sec:remind}
We shall also prove and use the following proposition, which is inspired by the work of \cite{Grunewalder18} (Proposition 3.2).
\begin{proposition}\label{prop:emp}
	Let $\Br \subseteq \Br_{\S}$ and let $\F_\Br := \{ f \ind{A} : f \in \F, A \in \Br \}$ where $\F$ is a set of functions in $\L ^1(\Q)$ uniformly bounded. If $\Br$ and $\F_\Br$ are $\Q$-Donsker classes then for any $\alpha \in [0, 1/2)$ and with $\Br_n := \{A \in \Br, \Q_n(A) \geq n^{-\alpha}\}$ we have
	\begin{equation*}
	\sup \limits_{f \in \F} \sup \limits_{A \in \Br_n} \big| \E_n \left[ f \mid A \right] - \E \left[f \mid A \right] \big| = O_{\Pr^*}(n^{\alpha - 1/2}).
	\end{equation*}
\end{proposition}

\begin{corollary}\label{prop:expectation}
	Let $\Br \subseteq \Br_\S$ be a $\Q$-Donsker class. If $Y$ is bounded then for any $i \in \Nset$ and any $\alpha \in [0, 1/2)$, with $\Br_n := \{A \in \Br, \Q_n(A) \geq n ^{-\alpha}\}$ we have
	\begin{equation}\label{eq:expectation}
	\sup \limits_{A \in \Br_n} \left| \E_n \left[ Y^i \mid (\X , Y) \in A \right] - \E \left[ Y^i \mid (\X , Y) \in A \right] \right| = O_{\Pr^*}(n^{\alpha - 1/2}),
	\end{equation}
	and 
	\begin{equation}\label{eq:variance}
	\sup \limits_{A \in \Br_n} \left| \V_n \left[ Y \mid (\X , Y) \in A \right] - \V\left[ Y \mid (\X , Y) \in A \right] \right| = O_{\Pr^*}(n^{\alpha - 1/2}).
	\end{equation}
\end{corollary}
Proofs of these results are deferred to the appendix.

It seems that the result of Corollary \ref{prop:expectation}, which is of independent interest, does not appear as such in the existing literature. As a first application of Corollary \ref{prop:expectation}, we show that any sequence of partitions with shrinking cell diameters is a suitable covering. We define the diameter of a cell $\r$ as $\text{Diam}(\r)=\sup_{\x\in \r,\,\x'\in \r} \|\x-\x'\|$, where $\|\cdot\|$ is any norm of $\Rset^d$.

\begin{proposition}\label{prop:diam}
	Consider a sequence $(\P_n)_{n\in\Nset}$ of data-dependent partitions that satisfies the coverage condition \ref{H:subsetcoverage} with $\alpha\in[0,\frac 12)$ and such that $$\bigcup_{n\in\Nset^*}\bigcup_{\d_n\in(\Rset^d)^n}\{\r\times\Rset : \r\in\P_n(\d_n)\}$$ is a.s. a $\Q$-Donsker class. Suppose that $\sigma^2$ is known, $g^*$ is uniformly continuous and
	\begin{equation}\label{eqmax}
		\max_{\r\in\P_n}\text{Diam}(\r)=o_\Pr(1).
	\end{equation}
	Then the sequence $(\P_n)$ is $\alpha$-suitable.
\end{proposition}

\begin{proof}
	Let us show that each cell is either significant or insignificant. Thanks to Condition~\ref{H:subsetcoverage}, Corollary~\ref{prop:expectation} (Eq.~\eqref{eq:variance}) and Remark~\ref{rem:P},
	\begin{equation}\label{eq:mvn}
	\max_{\r\in\P_n}| \V_n(Y | \X \in \r ) - \V(Y \mid \X \in \r ) | =O_{\Pr}(n^{\alpha-1/2}).
	\end{equation}
	Moreover $\V(Y \mid \X \in \r) = \V(g^*(\X)\mid \X\in \r) +\sigma^2$. Thus, as the redundancy condition \ref{H:redundancy} is automatically satisfied for cells of a partition, the desired result will follow if we check that $\varepsilon_n = o_\Pr(1)$ with 
	\begin{equation*}
	\varepsilon_n := \max \limits_{\r \in \P_n} \sqrt{(\V_n(Y | \X \in \r ) - \sigma^2)_+}.
	\end{equation*}
	
	From~\eqref{eq:mvn} we remark that
	\begin{equation*}
	\varepsilon_n \le \max \limits_{\r \in \P_n} \sqrt{\V(g^*(\X) \mid \X \in \r )}+O_{\Pr}(n^{\alpha/2-1/4})\;.
	\end{equation*}
	
	For all $n$, if $\r\in\P_n$, then $\r\times \Rset\in\Br_\S$.
	We denote $\X_\r$ and $\X_\r'$ two independent variables distributed as $\X$ given that $\X \in \r$. We obtain
	\begin{align*}
	\V(g^*(\X) \mid \X \in \r )&=\V(g^*(\X_\r))\\
	&=\tfrac12\V\big(g^*(\X_\r)-g^*(\X_\r')\big)\\
	&\le\tfrac12\E\left[(g^*(\X_\r)-g^*(\X_\r'))^2\right] \,.
	\end{align*}
	Thus, if we denote $w$ the modulus of continuity of $g^*$, we get
	$$ \sqrt{\V(g^*(\X) \mid \X \in \r )}\le 2^{-1/2}w(\text{Diam}(\r)).$$
	By uniform continuity, the condition~\eqref{eqmax} implies that 
	$$ \varepsilon_n\le 2^{-1/2} \max_{\r\in\P_n}\big(w(\text{Diam}(\r))\big)+O_\Pr(n^{\alpha/2-1/4})=o_\Pr(1)\,.$$ 
	Thus, from~\eqref{eq:Cni_emp}, each cell which is not significant is insignificant and the corresponding covering sequence is $\alpha$-suitable.
\end{proof}

%

\subsection{Estimation-approximation decomposition}

			Let $\Gset_n$ denote the set of piecewise constant functions on the partition $\P(\C_n) \cup \{\c_n^c\}$ such that $\forall g\in\Gset_n, \forall \x\in\Rset^d, |g(\x)|\leq L$ and $\forall g\in\Gset_n, \forall \x\in\c_n^c, g(\x) = 0$.
		
	%
		
		The excess risk $\ell(g^*,g_n)$ can be decomposed into two terms  mimicking  Lemma 10.1 of \cite{DistributionFree}. First notice that under the conditions of   Theorem~\ref{th:consistance} on $g_n$ we have
		\begin{equation}\label{eq:gnERMGn}
			\frac 1n \sum_{i=1}^n (g_n(\X_i) - Y_i)^2 \ind{\c_n}(\X_i) \leq \frac 1n \sum_{i=1}^n (g(\X_i) - Y_i)^2 \ind{\c_n}(\X_i), \qquad \forall g\in\Gset_n.
		\end{equation}
		Now, since moreover $g_n\in\Gset_n$, 
		\begin{align*}
			\ell (g^*, g_n) 	&= \E\bigl[ \bigl(g_n(\X) - g^*(\X)\bigr)^2 \bigr] \\
						&= \E\bigl[ \bigl(g_n(\X) - Y)^2 \bigr] - \E \bigl[ \bigl(g^*(\X) - Y\bigr)^2 \bigr] \\
						&= \E\bigl[ \bigl(g_n(\X) - Y)^2 \bigr] - \inf_{g\in\Gset_n} \E \bigl[ \bigl(g(\X) - Y\bigr)^2 \bigr] \\
						&\quad + \inf_{g\in\Gset_n} \E \bigl[ \bigl(g(\X) - Y\bigr)^2 \bigr] - \E \bigl[ \bigl(g^*(\X) - Y\bigr)^2 \bigr] \\
						&= \sup_{g\in\Gset_n} \biggl\{ \E\bigl[ \bigl(g_n(\X) - Y)^2 \bigr] - \frac 1n \sum_{i=1}^n (g_n(\X_i) - Y_i)^2 \\
						&\quad+ \frac 1n \sum_{i=1}^n (g_n(\X_i) - Y_i)^2 - \frac 1n \sum_{i=1}^n (g(\X_i) - Y_i)^2 \\
						&\quad + \frac 1n \sum_{i=1}^n (g(\X_i) - Y_i)^2 - \E \bigl[ \bigl(g(\X) - Y\bigr)^2 \bigr] \biggr\} \\
						&\quad + \inf_{g\in\Gset_n} \E \bigl[ \bigl(g(\X) - g^*(\X)\bigr)^2 \bigr] \\
						&\leq 2 \sup_{g\in\Gset_n} \bigl|\E\bigl[ \bigl(g(\X) - Y)^2 \bigr] - \frac 1n \sum_{i=1}^n (g(\X_i) - Y_i)^2\bigr] \bigr|\\
						&\quad + \underbrace{\frac 1n \sum_{i=1}^n (g_n(\X_i) - Y_i)^2\ind{\c_n^c}(\X_i)}_{\leq 4L^2 \Q_n(\c_n^c)} \\
						&\quad + \inf_{g\in\Gset_n} \E \bigl[ \bigl(g(\X) - g^*(\X)\bigr)^2 \bigr].
		\end{align*}
		Up to the term $4L^2 \Q_n(\c_n^c)$ which converges to zero by assumption, this is the standard decomposition of the risk into the estimation error and the approximation error. To prove the theorem it is sufficient to prove that
		\begin{equation}\label{proof:approx}
			\inf \limits_{g \in \Gset_n} \E \left[ \left(g(\X) - g^*(\X) \right)^2 \right] = o_\Pr(1)
		\end{equation}
		and 
		\begin{equation}\label{proof:estim}
			\sup_{g\in\Gset_n} \bigl|\E\bigl[ \bigl(g(\X) - Y)^2 \bigr] - \frac 1n \sum_{i=1}^n (g(\X_i) - Y_i)^2\bigr] \bigr| = o_\Pr(1).
		\end{equation}

\subsection{Approximation error}

		Let us define $\C_n^{i\setminus s} = \{ \r\setminus\c_n^s : \r\in\C_n^i \}$ where $\c_n^s = \cup_{\r \in \C_n^s} \r$. Notice that $\P(\C_n^{i\setminus s}) = \Bigl\{ A\setminus\c_n^s : A\in\P(\C_n^i) \Bigr\}$. Indeed, by Proposition~\ref{lemma:PC}:
	\begin{align}\label{eq:PCni-s}
		\P(\C_n^{i\setminus s}) &= \Bigl\{ \bigcap_{\r \in \tilde\C} \r \setminus \bigcup_{\r \in \C_n^{i\setminus s}\setminus\tilde\C} \r : \tilde\C \subseteq \C_n^{i\setminus s} \Bigr\} \nonumber\\
		&= \Bigl\{ \bigcap_{\r \in \tilde{\C^i}} (\r\setminus\c_n^s) \setminus \bigcup_{\r \in \C_n^i\setminus\tilde{\C^i}} (\r\setminus\c_n^s) : \tilde{\C^i} \subseteq \C_n^i \Bigr\} \nonumber\\
		&= \Bigl\{ \bigl(\bigcap_{\r \in \tilde{\C^i}} \r \setminus \bigcup_{\r \in \C_n^i\setminus\tilde{\C^i}} \r\bigr)\setminus\c_n^s : \tilde{\C^i} \subseteq \C_n^i \Bigr\} \nonumber\\
		&= \Bigl\{ A\setminus\c_n^s : A\in\P(\C_n^i) \Bigr\}.
	\end{align}
		So that, by Proposition~\ref{lemma:PC} again, $\P(\C_n)$ is a partition of $\c_n$ finer than $\P(\C_n^s)\cup\P(\C_n^{i\setminus s})$. Hence, with 
		\begin{equation*}
			\tilde g_n =  \mathop{\sum_{A\in\P(\C_n^s)\cup\P(\C_n^{i\setminus s})}}_{\Q(A)>0} \E[Y | \X\in A]\ind{A},
		\end{equation*}
		we have $\tilde g_n\in \Gset_n$ and 
		\begin{align*}
			\inf_{g\in\Gset_n} \sqrt{\E \left[ \left(g(\X) - g^*(\X) \right)^2 \right]} &\leq \sqrt{\E \left[ \left(\tilde g_n(\X) - g^*(\X) \right)^2 \right]} \\ 
			&\leq \sqrt{\E \left[ \left(\tilde g_n(\X) - g^*(\X)\ind{\c_n}(\X) \right)^2 \right]}\\
			&\quad + \underbrace{\sqrt{\E \left[ g^*(\X)^2\ind{\c_n^c}(\X) \right]}}_{\leq L\sqrt{\Q(\c_n^c)}}.
		\end{align*}	
		Thus, to prove~\eqref{proof:approx}, it suffices to show that $\W_n = o_\Pr(1)$ where
		\begin{equation*}
		\W_n := \E\left[\left(\tilde g_n(\X)-g^*(\X)\ind{\c_n}(\X)\right)^2\right].
		\end{equation*}
			
		We will repeatedly make use of the following equalities: 
		\begin{equation}\label{eq:expectYg}
			\begin{cases}
				\E[ g^*(\X)\ind{\c_n}(\X) \mid \X\in\r ] = \E[ g^*(\X) \mid \X\in\r ] = \E[ Y \mid \X\in\r ]\,, \\ 
				\V(g^*(\X)\ind{\c_n}(\X) \mid \X\in\r ) = \V(g^*(\X) \mid \X\in\r)\,,\qquad 			\forall \r\subseteq\c_n 		
			\end{cases}	
		\end{equation}
		and
		\begin{equation}\label{eq:varianceYg}
 \V(Y | \X \in \r ) = \V(g^*(\X) | \X\in\r) + \sigma^2\,,\qquad		\forall \r\subseteq\Rset^d. 	
		\end{equation}
	
		Let us first remark that
		\begin{align}
			\W_n &= \E \left[ \left( \sum_{A\in\P(\C_n^s)\cup\P(\C_n^{i\setminus s})} \E\left[ Y \mid \X \in A \right] \ind{A}(\X) - g^*(\X)\ind{\c_n}(\X) \right)^2 \right] \nonumber\\
				& = \sum_{A' \in \P(\C_n^s)\cup\P(\C_n^{i\setminus s})\cup\{\c_n^c\}} \E\left[ \left( \sum_{A \in \P(\C_n^s)\cup\P(\C_n^{i\setminus s})} \E\left[ g^*(\X)\ind{\c_n}(\X) \mid \X \in A \right] \ind{A}(\X) \right.\right.\nonumber\\&\quad\left.\left.- g^*(\X)\ind{\c_n}(\X) \right)^2 \ind{A'}(\X)\right] \nonumber\\
				& = \sum_{A' \in \P(\C_n^s)\cup\P(\C_n^{i\setminus s})\cup\{\c_n^c\}} \E\left[ \left(\E\left[ g^*(\X)\ind{\c_n}(\X) \mid \X \in A' \right] - g^*(\X)\ind{\c_n}(\X) \right)^2 \ind{A'}(\X)\right] \nonumber\\
				& = \sum_{A' \in \P(\C_n^s)\cup\P(\C_n^{i\setminus s})\cup\{\c_n^c\}} \E\left[ \left. \left(\E\left[ g^*(\X)\ind{\c_n}(\X) \mid \X \in A' \right] - g^*(\X)\ind{\c_n}(\X) \right)^2 \right| \X \in A' \right]\nonumber\\
				& \quad \times\mathbb P \left(\X\in A'\right)\label{eq:wn}
		\end{align}
		which shows that $\W_n$ is a within-group variance for the variable $g^*(\X)\ind{\c_n}(\X)$ and the groups $\P(\C_n^s)\cup\P(\C_n^{i\setminus s})\cup\{\c_n^c\}$.
		
		According to the decomposition  
		\begin{multline*}
			\W_n = \underbrace{\sum_{A \in \P(\C_n^s) } \E\left[ \left. \left(\E\left[ Y \mid \X \in A \right] - g^*(\X)\ind{\c_n}(\X) \right)^2 \right| \X \in A \right] \mathbb P \left(\X\in A\right)}_{\W_n^s} \\
			+ \underbrace{\sum_{A \in \P(\C_n^{i\setminus s}) } \E\left[ \left. \left(\E\left[ Y \mid \X \in A \right] - g^*(\X)\ind{\c_n}(\X) \right)^2 \right| \X \in A \right] \mathbb P \left(\X\in A\right)}_{\W_n^{ i\setminus s}} \\
			+ \underbrace{\E\left[ \left. \left(\E\left[ g^*(\X)\ind{\c_n}(\X) \mid \X \in \c_n^c \right] - g^*(\X)\ind{\c_n}(\X) \right)^2 \right| \X \in \c_n^c \right] \mathbb P \left(\X\in \c_n^c\right)}_{0}, 
		\end{multline*}
		it is sufficient to prove that $\W_n^s\xrightarrow[n\to\infty]{\Pr}0$ and $\W_n^{ i\setminus s}\xrightarrow[n\to\infty]{\Pr}0$.
	
		To deal with $\W_n^s$, we start from the decomposition of the total variance into the within-group and the between-group variances:
		\begin{align}\label{eq:var_decomposition}
			\W_n^s &= \E[(g^*(\X)\ind{\c_n}(\X) - \E[g^*(\X)\ind{\c_n}(\X)])^2\ind{\c_n^s}(\X)] - \B_n^s\nonumber\\
			&\leq \V( g^*(\X)\ind{\c_n}(\X) ) - \B_n^s
		\end{align}
		where
		\begin{equation*}
			\B_n^s := \sum_{A \in \P(\C_n^s)} \left( \E \left[ g^*(\X)\ind{\c_n}(\X) \mid \X \in A \right] - \E\left[ g^*(\X)\ind{\c_n}(\X) \right] \right)^2\mathbb P \left(\X \in A\right).
		\end{equation*}
		The between group variance $\B_n^s$ will be lower estimated by a function of $\W_n^s$ from which an upper bound of $\W_n^s$ will follow. The key point of this lower bound is to use the definition of $\C_n^s$ in \ref{H:significance}. To make the terms controlled by this significant condition appear, we lower bound $\B_n^s$ by a sum over the elements of $\C_n^s$ instead of elements of $\P(\C_n^s)$ and the expectations is replaced by empirical expectations. After applying the inequality provided by the significant condition we'll have to do the reverse operation to make $\W_n^s$ appear again. 
		
		Let us first lower bound $\B_n^s$ by a sum over the elements of $\C_n^s$ instead of elements of $\P(\C_n^s)$ and replace the expectations by empirical expectations: 
		\begin{align}
			\B_n^s 	&= \sum_{\r \in \C_n^s} \sum_{A \in  \P_{\C_n^s}(\r) } \frac{1}{\# \varphi_{\C_n^s}(A) } \left(\E \left[ Y \mid \X \in A \right] - \E\left[ g^*(\X)\ind{\c_n(\X)} \right] \right)^2 \mathbb P \left(\X \in A\right) \nonumber\\
					&\geq \sum_{\r \in \C_n^s} \frac{1}{\mnsr} \sum_{A \in  \P_{\C_n^s}(\r) } \left(\E \left[ Y \mid \X \in A \right] - \E\left[ g^*(\X)\ind{\c_n(\X)} \right] \right)^2 \mathbb P \left(\X \in A\right) \nonumber\\
					&\ge \frac{1}{\mns} \sum_{\r \in \C_n^s}\sum_{A \in \P_{\C_n^s}(\r)} \left(\E \left[ Y \mid \X \in A \right] - \E\left[ g^*(\X)\ind{\c_n(\X)} \right] \right)^2\mathbb P \left(\X \in A\mid \X\in \r\right) \mathbb P \left( \X\in \r\right) \nonumber\\
					&\ge \frac{1}{\mns} \sum_{\r \in \C_n^s} \left(\sum_{A \in \P_{\C_n^s}(\r) } \E \left[ Y \mid \X \in A \right]\mathbb P \left(\X \in A\mid \X\in \r\right) - \E\left[ g^*(\X)\ind{\c_n(\X)} \right] \right)^2 \mathbb P \left( \X\in \r\right) \nonumber\\
					&= \frac{1}{\mns} \sum_{\r \in \C_n^s} \Bigl( \E \left[ Y \mid \X \in \r \right] - \E \left[ g^*(\X)\ind{\c_n(\X)} \right] \Bigr)^2 \;\Pr \left( \X\in \r\right) \nonumber \\
					&\ge \frac1\mns \sum_{\r\in\C_n^s}\left(V_{n,\r\times\Rset}^2\ - \Delta_n \right) \Pr(\X\in\r) \label{minoBnUn}
		\end{align}
		where we applied Jensen's inequality for the third to last inequality and where 
		\begin{equation*}
			\Delta_n := \sup_{A\in\Br_n}\{ V_{n,A}^2-U_{n,A}^2 \}	
		\end{equation*}
		with  for any $A\in\Br$,
		\begin{gather*}
			U_{n,A}:=\E \left[ Y \mid (\X,Y) \in A \right] - \E \left[ g^*(\X)\ind{\c_n}(\X) \right] \\
			V_{n,A}:=\E_n \left[ Y \mid (\X,Y) \in A \right] - \E_n \left[ Y \right]
		\end{gather*}
		and $\Br_n = \{A \in \Br \text{ s.t. } \Q_n(A) > n^{-\alpha} \}$.
		
		Continuing~\eqref{minoBnUn} with the definition of $\C_n^s$ in \ref{H:significance} in mind,
		\begin{align*}
			\B_n^s 	&\ge \frac1\mns \sum_{\r\in\C_n^s} \left(\beta_n^{-2}(\V_n(Y|\X\in\r)-\sigma_n^2\right) - \Delta_n\big) \; \Pr(\X\in\r) \\
					&\ge \frac1\mns \sum_{\r\in\C_n^s} \left(\beta_n^{-2}(\V_n(Y|\X\in\r)-\sigma^2\right) - \beta_n^{-2} |\sigma_n^2-\sigma^2| - \Delta_n\big) \; \Pr(\X\in\r).
		\end{align*}
		
		Since $\W_n^s = \sum_{A \in \P_{\C_n}(\C_n^s) } \V(g^*(\X)\ind{\c_n}(\X) | \X\in A)  \mathbb P \left(\X\in A\right)$, this last term can be lower bounded by $\W_n^s$ if the empirical variances are replaced by variances and the sum over $\C_n^s$ by a sum over $\P(\C_n^s)$. Let us then define 
		\begin{equation*}
			\Delta'_n:=\sup_{A\in\Br_n}\left\{\big|\V(Y | (\X,Y)\in A)-\V_n(Y | (\X,Y)\in A)\big|\right\}
		\end{equation*}
		and write 
		\begin{align*}
			\mathbb B_n^s 	&\ge \frac1\mns \sum_{\r\in\C_n^s} \left(\beta_n^{-2}(\V(Y \mid \X\in\r)-\sigma^2-\Delta'_n) - \beta_n^{-2} |\sigma_n^2-\sigma^2| - \Delta_n\right) \; \Pr(\X\in\r) \\
						&= \frac1\mns \sum_{\r\in\C_n^s} \left(\beta_n^{-2}(\V(g^*(\X) \ind{\c_n}(\X) \mid \X\in\r) - \Delta'_n) - \beta_n^{-2} |\sigma_n^2-\sigma^2| - \Delta_n\right) \; \Pr(\X\in\r) \tag*{by \eqref{eq:expectYg} and \eqref{eq:varianceYg}} \\
						&= \frac{\beta_n^{-2}}\mns \times
						\sum_{\r\in\C_n^s} \left(\E\left[ \big(\E\left[ Y \mid \X \in \r\right] - g^*(\X)\ind{\c_n}(\X) \big)^2 \mid \X \in \r \right]\right.\\
						&\quad\left.-(\Delta'_n + |\sigma_n^2-\sigma^2| + \beta_n^{2}\Delta_n)\right)\Pr(\X\in\r) \\
						&\geq \frac{\beta_n^{-2}}\mns \times \sum_{\r \in \C_n^s} \sum_{A \in \P_{\C_n^s}(\r) }\left( \E\left[ \big(\E\left[ Y \mid \X \in \r\right] - g^*(\X)\ind{\c_n}(\X) \big)^2 \mid \X \in A \right] \right) \mathbb \Pr(\X \in A) \\
						& \quad - (\beta_n^{-2}\Delta'_n + \beta_n^{-2}|\sigma_n^2-\sigma^2| + \Delta_n) \tag*{\text{ since } $\sum_{\r\in\C_n^s}\Pr(\X\in\r) \le M(\C_n^s)$} \\
						&\geq \frac{\beta_n^{-2}}\mns \times \sum_{\r \in \C_n^s} \sum_{A \in \P_{\C_n^s}(\r) }\left( \E\left[ \big(\E\left[ Y \mid \X \in A \right] - g^*(\X)\ind{\c_n}(\X) \big)^2 \mid \X \in A \right] \right) \mathbb \Pr(\X \in A) \\
						& \quad - (\beta_n^{-2}\Delta'_n + \beta_n^{-2}|\sigma_n^2-\sigma^2| + \Delta_n) \\
						&\geq \beta_n^{-2}\frac{m(\C_n^s)}\mns \times \sum_{A \in  \P(\C_n^s) }\left( \E\left[ \big(\E\left[Y \mid \X \in A\right] - g^*(\X)\ind{\c_n}(\X) \big)^2 \mid \X \in A \right] \right) \mathbb P(\X \in A) \nonumber\\
						& \quad - (\beta_n^{-2}\Delta'_n + \beta_n^{-2}|\sigma_n^2-\sigma^2| + \Delta_n) \nonumber\\
						&= \beta_n^{-2} \frac{m(\C_n^s)}{\mns}\W_n^s - (\beta_n^{-2}\Delta'_n +  \beta_n^{-2}|\sigma_n^2-\sigma^2| + \Delta_n).
		\end{align*} 
		
		Together with~\eqref{eq:var_decomposition}, this yields
		\begin{equation*}
			\W_n^s \leq \frac{\V(g^*(\X)\ind{\c_n}(\X)) + \beta_n^{-2}\Delta'_n + \beta_n^{-2}|\sigma_n^2-\sigma^2| + \Delta_n}{1+\beta_n^{-2} \frac{m(\C_n^s)}{\mns} } \cdot
		\end{equation*}
		
		Under \ref{H:subsetcoverage} and \ref{H:Donsker}, Corollary~\ref{prop:expectation} applies and we obtain
		\begin{align*}
		\Delta_n 	&= \sup_{A\in\Br_n}\{(V_{n,A}-U_{n,A})(V_{n,A}+U_{n,A})\} \\
				&\leq \sup_{A\in\Br_n}\Bigl\{\bigl| \E_n[Y \mid (\X,Y)\in A] - \E[Y \mid (\X,Y)\in A] \bigr|\\
				&\quad + \bigl| \E[g^*(\X)\ind{\c_n(\X)}] - \E_n[Y] \bigr|\Bigr\} \times 4L\\
				&\leq 4L\times\bigl(\sup_{A\in \Br_n} \bigl| \E_n[Y|(\X,Y)\in A] - \E[Y|(\X,Y)\in A] \bigr| \\&\quad+ \bigl| \E[g^*(\X)\ind{\c_n}(\X)] - \E[g^*(\X)] \bigr| + \bigl| \underbrace{\E[g^*(\X)]}_{\E[Y]} - \E_n[Y] \bigr| \bigr)\\
				&\leq 4L\times\bigl(\Ops(n^{\alpha-1/2}) + L\Q(\c_n^c) + O_\Pr(n^{-1/2}) \bigr) \\
				&= \Ops(n^{\alpha-1/2}) 
		\end{align*}
		since $\Q(\c_n^c) = \Ops(n^{\alpha-1/2})$. Corollary~\ref{prop:expectation} also yields
		\begin{equation*}
			\Delta'_n = \Ops(n^{\alpha-1/2}).
		\end{equation*}
		Since it is moreover assumed that $|\sigma_n^2-\sigma^2| = O_\Pr(n^{\alpha-\frac 12})$, \eqref{eq:redundancy_significant} of Condition~\ref{H:redundancy} and Remark \ref{rem:P} lead to
		\begin{equation*}
			\W_n^s \xrightarrow[n\to\infty]{\Pr} 0.
		\end{equation*}
	
		To deal with $\W_n^{i\setminus s}$, remember \eqref{eq:PCni-s} and write
		\begin{align*}
			\W_n^{ i\setminus s} &= \sum_{A \in \P(\C_n^{i\setminus s})} \E\left[ \left(\E\left[ Y \mid \X \in A \right] - g^*(\X)\ind{\c_n}(\X) \right)^2 \ind{A}(\X) \right] \\		
							&= \sum_{A \in \P(\C_n^i)} \E\left[ \left(\E\left[ Y \mid \X \in A\setminus\c_n^s \right] - g^*(\X)\ind{\c_n}(\X) \right)^2 \ind{A\setminus\c_n^s}(\X) \right] \\							&\leq \sum_{A \in \P(\C_n^i)} \E\left[ \left(\E\left[ Y \mid \X \in A \right] - g^*(\X)\ind{\c_n}(\X) \right)^2 \ind{A\setminus\c_n^s}(\X) \right] \\	
							&\leq \sum_{A \in \P(\C_n^i)} \E\left[ \left(\E\left[ Y \mid \X \in A \right] - g^*(\X)\ind{\c_n}(\X) \right)^2 \ind{A}(\X) \right] \\	
							&\leq \frac1{m(\C_n^i)}\sum_{\r \in \C_n^i} \sum_{A \in \P_{\C_n^i}(\r)} \E\left[ \left(\E\left[ Y \mid \X \in A \right] - g^*(\X)\ind{\c_n}(\X) \right)^2 \ind{A}(\X) \right] \\	
							&\leq \frac1{m(\C_n^i)}\sum_{\r \in \C_n^i} \sum_{A \in \P_{\C_n^i}(\r)} \E\left[ \left(\E\left[ Y \mid \X \in \r \right] - g^*(\X)\ind{\c_n}(\X) \right)^2 \ind{A}(\X) \right] \\	
							&\leq \frac1{m(\C_n^i)}\sum_{\r \in \C_n^i} \E\left[ \left(\E\left[ Y \mid \X \in \r \right] - g^*(\X)\ind{\c_n}(\X) \right)^2 \ind{\r}(\X) \right] \\	
							&= \frac1{m(\C_n^i)}\sum_{\r \in \C_n^i} \E\left[ \left. \left(\E\left[ Y \mid \X \in \r \right] - g^*(\X)\ind{\c_n}(\X) \right)^2 \right| \X \in \r \right] \mathbb P \left(\X\in \r\right) \\
							&= \frac{1}{m(\C_n^i)}\sum_{\r \in \C_n^i} \left( \V(Y | \X \in \r ) - \sigma^2\right) \mathbb P \left(\X\in \r\right) \tag*{by \eqref{eq:expectYg} and \eqref{eq:varianceYg}} \\
							&\leq \frac{1}{m(\C_n^i)}\sum_{\r \in \C_n^i} \left( \V_n(Y | \X \in \r ) - \sigma_n^2 + |\sigma_n^2-\sigma^2| + \Delta'_n\right) \mathbb P \left(\X\in \r\right) \\
							&\leq \frac{1}{m(\C_n^i)}\sum_{\r \in \C_n^i} \left(\varepsilon^2_n + |\sigma_n^2-\sigma^2| + \Delta'_n\right)\mathbb P \left(\X\in \r\right) \tag*{under \ref{H:significance}} \\
		&\leq \frac{M(\C_n^i)}{m(\C_n^i)} \left(\varepsilon_n^2 + |\sigma_n^2-\sigma^2| + \Delta'_n\right) \xrightarrow[n\rightarrow\infty]{\Pr} 0 \tag*{under \eqref{eq:redundancy_insignificant} of Condition~\ref{H:redundancy} and with Remark~\ref{rem:P}.}
		\end{align*}

\subsection{Estimation error}

		The proof of \eqref{proof:estim} is inspired by \cite[Theorem 13.1]{DistributionFree} and its proof. It is not needed here to truncate the functions in $\Gset_n$ nor $Y$ since they are all bounded by assumption. Thus the assumptions \ref{H:partitions} do not need to involve the truncation constant.
			
		Recall that $\Gset_n$ is the set of piecewise constant functions with values in $[-L,L]$ on the elements of the partition $\P(\C_n)\cup \{\c_n^c\}$ such that $\forall g\in\Gset_n, \forall \x\in\c_n^c, g(\x) = 0$. Then $\Gset_n$ is a subset of
		\begin{equation*}
			\Gset_c \circ \Pi_n := \left\{g: \Rset^d \to \Rset : g = \sum_{A \in \P} f_A \ind{A}, \P \in \Pi_n, f_A \in \Gset_c \right\},
		\end{equation*}
		where $\Pi_n$ is a refinement of $\{\P(\C_n(\d_n)) : \d_n\in\S^n\}$ (see Defintion \ref{def:refiement}) and $\Gset_c$ is the set of constant functions $\Rset^d \to [-L,L]$. Then 
		\begin{multline*}
			\sup \limits_{g \in \Gset_n } \biggl| \frac 1n \sum_{i=1}^n (g(\X_i) - Y_i)^2 - \E\bigl[ \bigl(g(\X) - Y)^2 \bigr] \biggr| \\ \leq \sup \limits_{g \in \Gset_c \circ \Pi_n } \biggl|  \frac 1n \sum_{i=1}^n (g(\X_i) - Y_i)^2 - \E\bigl[ \bigl(g(\X) - Y)^2 \bigr] \biggr|,
		\end{multline*}
			
		According to \cite[Theorem 9.1 and Problem 10.4]{DistributionFree} we have,
		\begin{multline}\label{estim_bound}
			\Pr \left\{ \sup_{g \in \Gset_c \circ \Pi_n } \biggl|  \frac 1n \sum_{i=1}^n (g(\X_i) - Y_i)^2 - \E\bigl[ \bigl(g(\X) - Y)^2 \bigr] \biggr| > \varepsilon \right\} \\
			\leq 8 \E\left[\mathcal{N}_1\left( \frac{\varepsilon}{32L}, \Gset_c \circ \Pi_n , \X_1^n\right)\right]\exp\left\{ \frac{-n \varepsilon^2}{128.(4L^2)^2}\right\},
		\end{multline}
		where $\X_1^n = \{\X_1, \dots, \X_n \}$. 
		
		Here $\mathcal{N}_1 \left(\varepsilon, \Gset_c \circ \Pi_n , \X_1^n \right)$ is the random variable corresponding to the minimal number $N \in \Nset$ such that there exist functions $g_1, \dots, g_N: \Rset^d \to [-L, L] $ with the property that for every $g \in \Gset_c \circ \Pi_n $ there is a $j \in \{1,...,N\} $ such that
		
		\begin{equation*}
		\frac{1}{n} \sum \limits_{i=1}^n \left| g(\X_i) - g_j(\X_i)\right| \leq \varepsilon.
		\end{equation*}
		
		This number is called the $\varepsilon$-covering number of $\Gset_c \circ \Pi_n $. It can be interpreted as the complexity of the class. 
		Then using \cite[Lemma 13.1]{DistributionFree} (it can be checked to hold in the current situation, in particular for a family of finite collections of disjoint sets instead of a family of partitions of $\Rset^d$) we have
		\begin{multline*}
		\mathcal{N}_1 \left(\frac{\varepsilon}{32L}, \Gset_c \circ \Pi_n, \X_1^n \right) \\
		\leq \widetilde\Delta_n(\Pi_n) \left\{ \sup_{z_1, \dots, z_m \in \{\X_1, \dots, \X_n \}, m \leq n} \mathcal{N}_1\left(\frac{\varepsilon}{32L}, \Gset_c, z_1^m\right) \right\}^{\widetilde\M(\Pi_n)},
		\end{multline*}
					
		According to \cite[Lemma 9.2]{DistributionFree} for any set of functions $\Gset$ and any sample $z_1^m$ we have
		\begin{equation*}
		\mathcal{N}_1 \left(\frac{\varepsilon}{32L}, \Gset, z_1^m \right) \leq \mathcal{M}_1 \left(\frac{\varepsilon}{32L}, \Gset, z_1^m \right),
		\end{equation*}
		where $\mathcal{M}_1\left(\varepsilon, \Gset, z_1^m\right)$ is the maximal $N \in \Nset$ such that there exist functions $g_1, \dots, g_N \in \Gset$ with
		\begin{equation*}
		\frac{1}{n} \sum \limits_{i=1}^m \left| g_j(z_i) - g_k(z_i)\right| \geq \varepsilon,
		\end{equation*}
		for all $1 \leq j < k \leq N$. It is called $L_1$ $\varepsilon$-packing of $\Gset$ on $z_1^m$. See \cite[Definition 9.4 (c)]{DistributionFree}.
		
		Now, from the definition of $\Gset_c$,
		\begin{equation*}
		\sup_{z_1, \dots, z_m \in \{\X_1, \dots, \X_n \}, m \leq n} \mathcal{M}_1\left(\varepsilon, \Gset_c, z_1^m \right) \leq \biggl\lceil \frac{2L}{\varepsilon} \biggr\rceil +1.
		\end{equation*}
		Finally, 
		\begin{equation}\label{eq:N1_bound}
		\sup_{z_1, \dots, z_m \in \{\X_1, \dots, \X_n \}, m \leq n} \mathcal{N}_1 \left(\frac{\varepsilon}{32L}, \Gset_c \circ \Pi_n, z_1^m \right) \leq \widetilde\Delta(\Pi_n) \biggl(\biggl\lceil \frac{64L^2}{\varepsilon} \biggr\rceil + 1\biggr)^{\widetilde\M(\Pi_n)}.
		\end{equation}
		According to \eqref{estim_bound} and \eqref{eq:N1_bound} we have:
		\begin{multline*}
		\Pr \left\{ \sup \limits_{g \in \Gset_n}\left| \frac{1}{n} \sum_{i=1}^{n} |g(\X_i) - Y_{i} |^2 - \E \left[ |g(\X) - Y|^2\right]\right| > \varepsilon \right\} \\
		\leq 8 \widetilde\Delta_n(\Pi_n) \biggl(\biggl\lceil \frac{64L^2}{\varepsilon} \biggr\rceil + 1\biggr)^{\widetilde\M(\Pi_n)} \exp\left(- \frac{n \varepsilon^2}{128. (4L^2)^2}\right)
		\end{multline*}
		and since
		\begin{align*}
		& 8\tilde \Delta_n(\Pi_n) \biggl(\biggl\lceil \frac{64L^2}{\varepsilon} \biggr\rceil + 1\biggr)^{\widetilde\M(\Pi_n)} \exp\left(- \frac{n \varepsilon^2}{2048L^4}\right) \\
		& = 8 \exp\left( - \frac{n}{L^4} \left( \frac{\varepsilon^2}{2048} - \frac{\log{\widetilde\Delta_n(\Pi_n)} L^4}{n} - \frac{ \widetilde\M(\Pi_n) L^4 \log\left( \biggl\lceil \frac{64L^2}{\varepsilon} \biggr\rceil + 1 \right)}{n} \right) \right),
		\end{align*}
		this concludes the proof of \eqref{proof:estim} and of Theorem~\ref{th:consistance}.

\section{Illustrations}\label{sec:Application}
In this section we propose a simple algorithm to generate data-dependent coverings using either Random Forests (RF, \cite{Breiman01}) or Gradient Boosted trees (GB, \cite{Friedman01}) or Stochastic Gradient Boosting trees (SGB, \cite{Friedman02}) as rule generator. The interest is twofold; First, it exhibits examples of data-dependent quasi-coverings that are likely to be suitable as in Definition \ref{def:suitable_covering}. Second, we observe that the interpretability indices  are small when applying our algorithm predicting real data selecting suitable sets of rules. Note that all rules considered here are based on intervals of $\Rset$ as in Remark \ref{rem:hyperrect} and then their length is automatically smaller than or equal to $d$.

\subsection{Covering Algorithm}\label{subsec:Algo}
The proposed algorithm generates an estimator based on a data-dependent quasi-covering:
\begin{enumerate}
	\item Generate trees with a given method among RF, GB or SGB with a maximal tree size, $tree\_size$, and a maximal number of generated rules (all the nodes and leaves of the trees), $max\_rules$.	
	\item For a chosen $\alpha \in (0,1/2)$, set $\beta_n = n^{\alpha/2 - 1/4}$ and $\varepsilon_n = \beta_n s_n$, where $s_n$ is the empirical standard deviation of $Y$. Keep all rules which length is  less than or equal to $l\_max \in \{1, \dots, d\}$ and which fulfill \ref{H:subsetcoverage}.
	\item Split this set of rules into two sets: The set of significant rules $S_n$ (\ie rules $\r$ that fulfill 	$ \beta_n\big| \E_n[Y | \X \in \r] - \E_n[Y]\big| \geq \sqrt{(\V_n(Y | \X \in \r ) - \sigma_n^2)_+}$) and the set of insignificant rules $I_n$ (\ie rules $\r$ that are not in $S_n$ and that fulfill $\varepsilon_n \geq \sqrt{ ( \V_n(Y | \X \in \r ) - \sigma_n^2)_+}$). 
	\item Select a minimum set of rules $\C_n$ using Algorithm \ref{pseudocode} given in the appendix. If at any step the set of selected rules forms a covering, the selection process is stopped. A rule is added to the currently selected set of rules if and only if it has at least a rate $1-\gamma\in (0,1)$ of points not covered by the current set of rules\footnote{This step was already described in Remark \ref{rem:redundancy} in order to fulfill the redundancy condition \ref{H:redundancy}.}. The set $S_n$, sorted by decreasing empirical coverage rate, is browsed first. Then, if necessary, the set $I_n$, sorted by increasing empirical variance, is browsed. 
\end{enumerate}
As explained in Definition~\ref{def:suitable_covering} the set $\C_n$ is the union of the significant set $\C_n^s = \C_n\cap S_n$ and the insignificant set $\C_n^i = \C_n\cap I_n$. In practice the set $\C_n^s$ is the most interesting one because it identifies the rules where the conditional mean is prominent. 

The datasets and the code for the illustrations are available on \href{https://github.com/VMargot/CoveringAlgorithm}{GitHub}. The code is written in both the Python and R languages.

\subsection{Artificial data}\label{sec:artif}
We consider here the same model as in \cite{Friedman08}. We generate $n=5000$ observations following the regression setting
\begin{equation*}
Y = g^*(\X) + Z,
\end{equation*}
where $d=100$ (the dimension of $\X$) and
\begin{multline}\label{data}
g^*(\X) = 9 \prod_{j=1}^3 \exp\left(-3\left(1- X_j\right)^2\right) - 0.8\exp\left(-2 \left(X_4 - X_5 \right)\right) \\
+2\sin^2(\pi \cdot X_6) - 2.5\left(X_7 - X_8\right),
\end{multline}
and $Z \sim \mathcal{N}(0,\sigma^2)$. The value of $\sigma>0$ was chosen to produce a two-to-one signal-to-noise ratio. The variables were generated from a uniform distribution on $\{0/10,\ldots,9/10\}$.
It is important to note that only the eight first variables are informative; the $92$ others are just noise. The coefficients that multiply each of the terms in $g^*$ have been chosen to ensure that the variables have approximately the same influence. 

We evaluate the accuracy of the estimators with the mean squared error defined by
$$MSE = \dfrac{\E_{\Q}\left[(Y - g_n(\X))^2\right]}{\V(Y)}.$$

In order to evaluate the error without the noise variance, we also consider the following criterion:
$$MSE^* = \dfrac{\E_{\Q}\left[(g^*(\X) - g_n(\X))\right]}{\V(g^*(\X) )}.$$

We approximate the criteria $MSE$ and $MSE^*$ with $50000$ test observations sampled independently from $\Q$.

\subsubsection{Execution}
We run $M=100$ simulations with $5000$ independent observations. For each simulation we compare the Covering Algorithm with RF as rules generator (Covering) with a classical Random Forest (RF) and with RuleFit \citep{Friedman08} with rules only.

RuleFit is a very accurate rule-based algorithm. First it generates a list of rules by considering all nodes and leaves of a boosted tree ensemble ISLE \citep{Friedman03}. Then rules are used as features in a sparse linear regression model obtained by Lasso \citep{Tibshirani96}.

For RF, we set the number of trees at $100$. For RuleFit and Covering we set the maximal number of generated rules at $4000$ and the maximal length of a rule is fixed at $k_{max} = 3$. And we set $\alpha = 1/2 - 1/100$ and $\gamma = 0.90$ for Covering.
	
\subsubsection{Results}
The $MSE$ and $MSE^*$ and the interpretability index \eqref{eq:inter} for each algorithm across the experiments are summarized in Table~\ref{tab:resume_exp} and plotted in Figure~\ref{fig:MSE}. 

\begin{figure}[h!]
	\centering
	\includegraphics[scale=0.35]{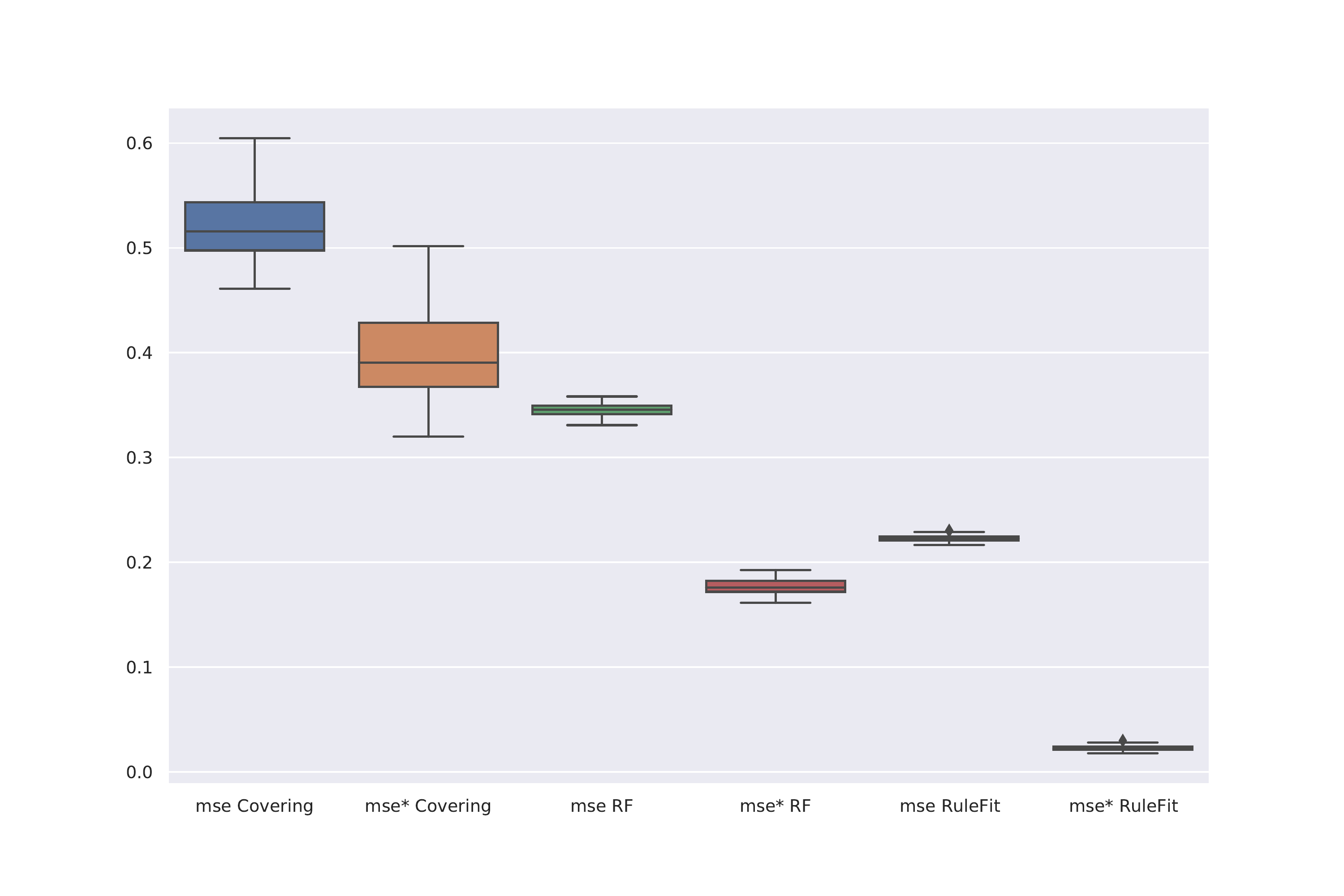}
	\caption{\label{fig:MSE} $MSE$ on $100$ realizations from model~\eqref{data} for Covering Algorithm (Covering), Random Forest (RF) and RuleFit.}
\end{figure}

\begin{table}[h!]
	\centering
	\begin{tabular}{crrrr}
		\hline
		{\makecell{\textbf{Random}\\ \textbf{Forest}}} & Nb rules &  Interpretability  & MSE & MSE* \\
		\hline
		mean  &  630718.78 &  6871475.52 &    0.35 &     0.18 \\
		std   &     429.56 &    28124.49 &    0.01 &     0.01 \\
		min   &  629458.00 &  6815078.00 &    0.33 &     0.16 \\
		25\%   &  630434.00 &  6848260.00 &    0.34 &     0.17 \\
		50\%   &  630701.00 &  6871049.00 &    0.35 &     0.18 \\
		75\%   &  631020.50 &  6890434.00 &    0.35 &     0.18 \\
		max   &  631790.00 &  6947934.00 &    0.36 &     0.19 \\
		\hline
	\end{tabular}
	\begin{tabular}{crrrr}
		\\
		\hline
		{\makecell{\textbf{Covering}\\ \textbf{Algorithm}}} & Nb rules &   Interpretability   &  MSE &  MSE* \\
		\hline
		mean  &        15.89 &      43.16 &          0.52 &           0.40 \\
		std   &         2.53 &       8.65 &          0.03 &           0.04 \\
		min   &        10.00 &      23.00 &          0.46 &           0.32 \\
		25\%   &        14.00 &      36.00 &          0.50 &           0.37 \\
		50\%   &        16.00 &      42.00 &          0.52 &           0.39 \\
		75\%   &        17.00 &      48.00 &          0.54 &           0.43 \\
		max   &        24.00 &      70.00 &          0.60 &           0.50 \\
		\hline
	\end{tabular}
	\begin{tabular}{crrrr}
		\\
		\hline
		{\textbf{RuleFit}} & Nb rules  &  Interpretability & MSE &  MSE* \\
		\hline
		mean  &         360.65 &      1155.63 &         0.22 &          0.02 \\
		std   &          55.30 &       232.09 &         0.00 &          0.00 \\
		min   &         253.00 &       747.00 &         0.22 &          0.02 \\
		25\%   &         314.75 &       965.25 &         0.22 &          0.02 \\
		50\%   &         358.50 &      1143.00 &         0.22 &          0.02 \\
		75\%   &         397.00 &      1297.25 &         0.22 &          0.02 \\
		max   &         507.00 &      1792.00 &         0.23 &          0.03 \\
		\hline
	\end{tabular}
	\caption{\label{tab:resume_exp}Number of rules generated by the algorithm (Nb Rules), interpretability index \eqref{eq:inter} (Interpretability) and mean squared errors $MSE$ and $MSE^*$ for each algorithm.}
\end{table}

\begin{figure}[h!]
	\centering
	\includegraphics[scale=0.30]{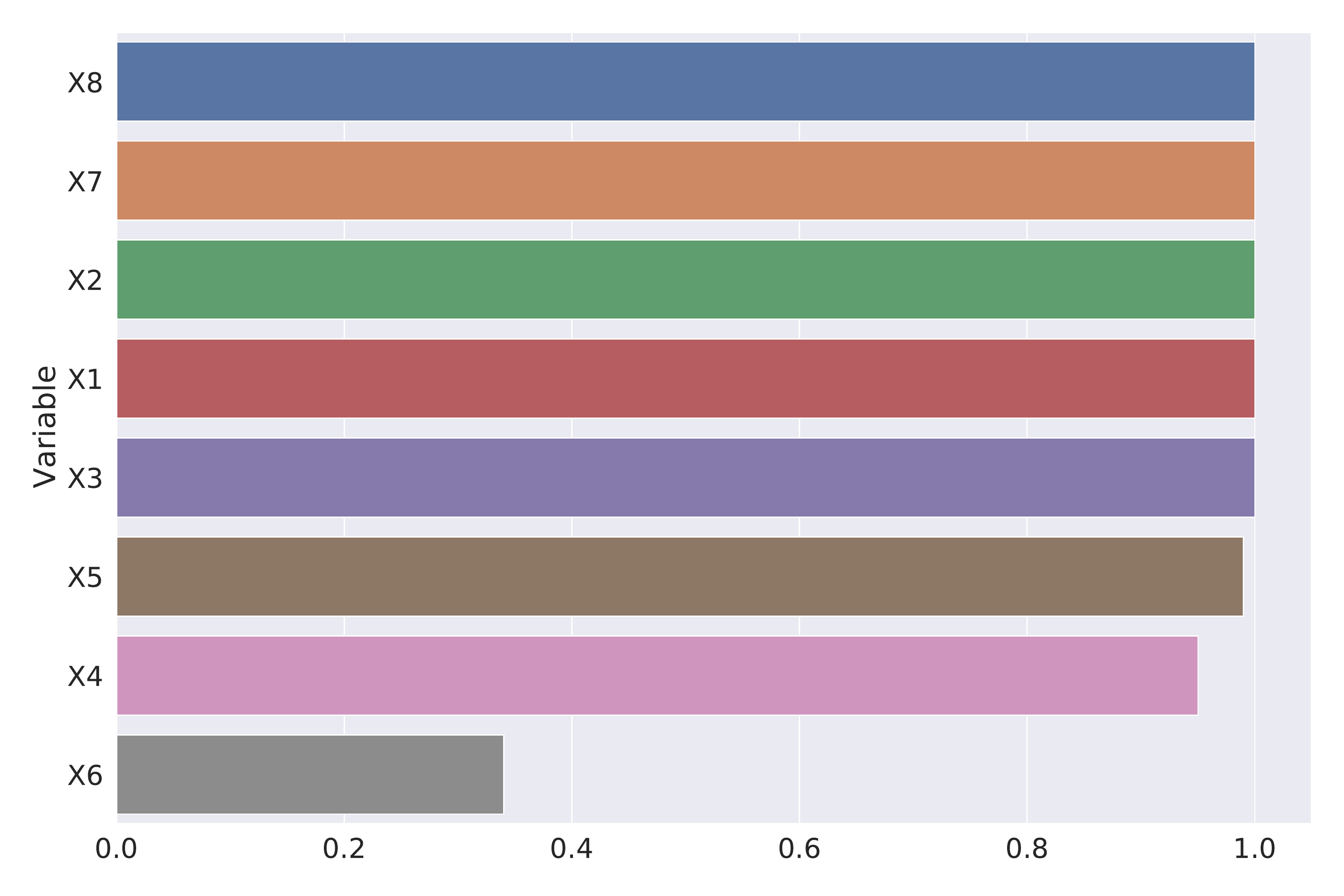}
	\caption{\label{fig:proba_occur}For each informative variable, frequency of occurrence in at least one rule selected by the Covering Algorithm in $100$ independent simulations.}
\end{figure}


\subsubsection{Comments}
Figure~\ref{fig:proba_occur} shows the frequency of occurrence of a variable in at least one rule of the selected set of rules. Note that only informative variables are involved in the selected rules. Moreover, in one thirds of the experiments, all informative variables are identified. It means the support of $g^*$ is well identified in these cases. In the most of the experiences, the variable $X6$ is the only variable that is not involved. One reason could be that RF is a random generator of rules that may not capture the importance of $X6$ at every run. The problem could be solved by considering a deterministic rule generator algorithm. 

Table~\ref{tab:resume_exp} emphasizes that Covering Algorithm generates more interpretable models than RF and RuleFit according to the interpretability index defined in \eqref{eq:inter}. Indeed, for the same constraints on the maximal number of rules and the maximal length of the rules, Covering Algorithm selects far fewer rules than RuleFit. Nevertheless, RuleFit is much more accurate than Covering Algorithm. The accuracy of RuleFit is high even if many noise variables are included in its generated model. Therefore, a posteriori analysis of the importance of the variables and rules of the generated model of RuleFit is crucial, see Section 9.1.2 of \cite{Friedman08} for more information.


\color{black}
\subsection{Real data}\label{sec:real}
For this application, we consider seven public datasets presented in Table~\ref{tab:reg_datasets}. with different dimension and number of observations. For the dataset \emph{Student} we have removed the variables $G1$ and $G2$ that are the first and the second grade respectively because the target attribute $G3$ has a strong correlation with $G2$ and $G1$. In \cite{Cortez08} the authors state that despite the complexity of the regression model, the prediction of $G3$ without $G2$ and $G1$ is much more useful in practice.

	\begin{table}[ !ht]
	\begin{center}
	\begin{tabular}{|l||c|p{7.5cm}|}
		\hline
		Name & $(n \times d)$ & Short description \\
		\hline
		Diabetes & $443 \times 10$ & Prediction of quantitative measure of disease progression one year after baseline \citep{Efron04}.\\[0.2cm]
		Prostate &$97 \times 8 $ & Prediction of the level of prostate-specific antigen based on clinical measures in men who were about to receive a radical prostatectomy \citep{ElementST}.\\[0.2cm]
		Ozone & $330 \times 9 $ & Prediction of atmospheric ozone concentration from daily meteorological measurements \citep{ElementST}. \\[0.2cm]
		Machine & $209 \times 8$ & Prediction of published relative performance \citep{Dua2019}. \\[0.2cm]
		MPG & $398 \times 8$ & Prediction of city-cycle fuel consumption in miles per gallon \citep{Dua2019}. \\[0.2cm]
		Boston & $506 \times 13$ & Prediction of the median price of neighborhoods, \citep{Harrison78}. \\[0.2cm]
		Student & $649 \times 32$ & Prediction of the final grade of students based on attributes collected by reports and questionnaires \citep{Cortez08}.\\[0.2cm]
		\hline
 	\end{tabular}
 	\end{center}
 	\caption{\label{tab:reg_datasets} Brief description of the public regression datasets.}
 	\end{table}

\subsubsection{Execution}
For each dataset we run $20$ executions. For each execution, the data are randomly split into a training set and a test set, with $70\%$ and $30\%$ ratios respectively. As a baseline of accuracy we consider RF \citep{Breiman01}  and as a baseline of interpretability we consider the CART algorithm \citep{CART}. We compare Covering Algorithm with RF as rule generator (CA RF),  GB trees as rule generator (CA GB),   SGB trees as rule generator (CA SGB) with RuleFit \citep{Friedman08} (with rules only), Node harvest \citep{Meinshausen10} (Nh) and SIRUS \citep{Benard20}. These algorithms  are the main existing interpretable rule-based algorithms for regression settings that are based on tree ensembles.

Node harvest uses also a tree ensemble as rule generator. The algorithm considers all nodes and leaves of RF as rules and solves a linear quadratic problem to fit a weight to each rule. So the prediction is a convex combination of rules.

SIRUS (Stable and Interpretable RUle Set) is designed as a stable predictive algorithm. It uses a custom RF algorithm to generate many rules (considering nodes and leaves). Then it selects the rules with a rate of occurrence greater than a tuning parameter $p_0$. To ensure a large number of occurrences on rules, the features are discretized.

The maximal number of leaves of CART is set at $tree\_size = 10$ and the number of tree for RF is set at $nb\_tree=1000$. For CA and RuleFit we set similarly the maximal number of rules generated by the rule generator $max\_rules = 4000$ and the size of a tree $tree\_size= 8$.   For CA, Node harvest and SIRUS we  set identically the maximum rule length $l\_max = 3$. For CA we also set $\alpha = 1/2 - 1/100$ and $\gamma = 0.90$. In these applications the real value of $\sigma^2$ is unknown. We estimate it by $\sigma_n^2$ the minimal variance of the generated rules fulfilling the covering condition \ref{H:subsetcoverage}. We have no guarantee that this estimator is good enough (see Remark \ref{rem:sigmaknown}). For SIRUS, the hyperparameter $p0$ is estimated by $10$-fold cross-validation and the maximal number of rules is set at $25$. The parameters setting is summarized in Table~\ref{tab:algo_params}.

\begin{table}[ !ht]
	\begin{center}
	\begin{tabular}{|l|p{7.7cm}|}
		\hline
		Algorithm & Parameters \\
		\hline
		CART & $tree\_size = 10$.\\
		RF & $nb\_tree=1000$.\\
		RuleFit & $tree\_size=8$, \newline $max\_rules=4000 $. \\ 
		Node harvest (Nh) & $l\_max=3$. \\
		SIRUS & $l\_max=3$, \newline $max\_selected\_rules=25$, \newline $ncv=10$. \\ 
		CA  & $max\_rules=4000$,\newline $tree\_size=8$,\newline $alpha=1/2-1/100$, \newline$gamma=0.90$, \newline $l\_max = 3$. \\[0.2cm]
		\hline
 	\end{tabular}
 	\end{center}
 	\caption{\label{tab:algo_params} Algorithm parameter settings.}
 	\end{table}

\subsubsection{Results}\label{sec:results}
Experimental results are gathered in Table~\ref{tab:res_reg}. For each dataset and algorithm we present an average over the 20 executions of the number of selected rules (Nb rules), the empirical coverage rate on the training set (Cov), the interpretability index \eqref{eq:inter} (Int) and the mean squared error on the test set divided by the empirical variance of the target (MSE). 

\begin{table}[ !ht]
	\begin{center}
 		\begin{tabular}{|l|c|c|c||c|c|c|}
			\hline
			\multirow{2}{*}{Dataset} & \multicolumn{3}{|c||}{Regression Tree (CART)} & \multicolumn{3}{|c|}{Random Forest (RF)} \\
			\cline{2-7}
			 & Nb rules & Int & MSE & Nb rules & Int & MSE \\
			\hline
			Diabetes & 10 & 30.2 & 1.14 & 191190.6 & 1070454.75 & \textbf{0.58} \\
			Prostate & 10 & 25.3 & 1.11 & 42047.95 & 169328.6 & \textbf{0.56}\\
			Ozone & 10 & 32 & 0.5 & 108082.95 & 572067.75 & \textbf{0.28}\\
			Machine & 10 & 24.2 & 0.18 & 73754.25 & 319376.45 & \textbf{0.06}\\
			MPG & 10 & 29.0 & 0.27 & 148954.4 & 623266.0 & \textbf{0.14}\\
			Boston & 10 & 26.3 & 0.29 & 212961.05 & 1270252.25 & \textbf{0.14} \\
			Student & 10 &  37.75 &  1.55 &  199278.65 & 1838107.25 & \textbf{0.74}\\
			\hline
		\end{tabular}
		\vspace*{0.3cm}
		
		\begin{tabular}{|l|c|c|c|c||c|c|c||c|}
			\hline
			\multirow{2}{*}{Dataset}  & \multicolumn{4}{|c||}{CA RF} & \multicolumn{4}{|c|}{CA GB} \\
			\cline{2-9}
			 & Nb rules & Cov & Int & MSE &  Nb rules & Cov & Int & MSE \\
			\hline
			Diabetes & \textbf{7.6} & 0.99 & \textbf{17.95} & 0.70 & 9.1 & 0.99 & 19.7 & 0.74 \\
			Prostate & \textbf{5.8} & 1.0 & \textbf{13.95} & 0.85 & \textbf{6.05} & 1.0 & \textbf{13.35} & 0.86 \\
			Ozone & \textbf{3.3} &  0.99 &  \textbf{5.85} & 0.42 & 3.95 &  0.99 &  6.55 & 0.39\\
			Machine & \textbf{2.85} &  0.99 &  \textbf{5.05} & 0.40 & 3.25 &  0.99 & 6.05 &  0.50\\
			MPG & \textbf{3.1} & 0.99 &  \textbf{4.3} & 0.28 &  \textbf{3.65} & 0.99 & 6.85 & 0.30\\
			Boston & \textbf{4.9} & 0.99 & \textbf{10.9} & 0.45 & 6.8 & 0.99 & 15.15 & 0.45\\
			Student & 18.45 & 0.99 & 51.4 & 0.86 & 32 & 0.98 & 86.75 &  1.01\\
			\hline
		\end{tabular}
		\vspace*{0.3cm}
				
		\begin{tabular}{|l|c|c|c|c||c|c|c|c|}
			\hline
			\multirow{2}{*}{Dataset}& \multicolumn{4}{|c||}{CA SGB}  & \multicolumn{4}{|c|}{RuleFit}\\
			\cline{2-9}
			 & Nb rules & Cov & Int & MSE & Nb rules & Cov & Int & MSE \\
			\hline
			Diabetes & 8.75 & 0.99 & 19.8 & 0.67 & 249.35 & 1.0 & 1168.7 & 0.71 \\
			Prostate & 6.5 & 1.0 & 15.05 & 0.76 & 78.4 & 1.0 & 268.8 & 0.66\\
			Ozone & 4.05 & 0.99 & 6.5 & 0.38 & 187.45 & 1.0 & 851 & \textbf{0.28} \\
			Machine & 3.2 & 0.99 & 6.1 & 0.52 & 99.85 & 1.0 & 278.15 & 0.08 \\
			MPG & 3.75 & 0.99 & 7.25 & 0.29 & 174.7 & 1.0 & 745.15 & \textbf{0.13} \\
			Boston & 6.35 & 0.99 & 13.85 & 0.46 & 264.95 & 1.0 & 1207.65 & \textbf{0.12} \\
			Student & 37.3 & 0.98 & 100.05 & 1.01 & 277.95 & 1.0 & 1280.4 & 0.90 \\
			\hline
		\end{tabular}
		\vspace*{0.3cm}
		
		\begin{tabular}{|l|c|c|c|c||c|c|c|c|}
			\hline
			\multirow{2}{*}{Dataset}  & \multicolumn{4}{|c||}{NodeHarvest} & \multicolumn{4}{|c|}{SIRUS} \\
			\cline{2-9}
			 & Nb rules & Cov & Int & MSE &  Nb rules & Cov & Int & MSE \\
			\hline
			Diabetes & 227.15 & 1.0 & 613.05 & \textbf{0.59} & 21.25 & 0.99 & 32.35 & \textbf{0.57} \\
			Prostate  & 66 & 1.0 & 160.35 & \textbf{0.57} & 23.25 & 0.99 & 33.7 & \textbf{0.59} \\
			Ozone & 181.45 & 1.0 & 472.9 &  \textbf{0.29} & 24.9 &  1.0 & 39.4 & 0.32 \\
			Machine & 67.15 & 1.0 & 159.0 &  0.32 & 13.9 & 1.0 & 21.85 &  0.26\\
			MPG & 100.15 & 1.0 & 257.85 & 0.15 &  25.0 & 1.0 & 45.2 &  0.21\\
			Boston & 139.85 & 1.0 &  345.4 & 0.20 &  25.0 & 0.99 & 38.5 & 0.25\\
			Student & 78.5 & 1.0 & 228.45 &  \textbf{0.76} & \textbf{15.35} & 1.0 & \textbf{26.2} & \textbf{0.77}\\
			\hline
		\end{tabular}
		
	\end{center}
 	\caption{\label{tab:res_reg} Average of the number of rules (NB rules), the empirical coverage rate on the training set (Cov), the interpretability index (Int) and the ratio mean squared error/empirical variance of the target (MSE) over 20 executions of usual interpretable algorithms for various regression public datasets. Best values are in bold, as well as values within $10\%$ of the best for each dataset.}
 	\end{table}

These results emphasize that suitable data-dependent quasi-coverings (see Definition \ref{def:suitable_covering}) are very efficient to generate an interpretable rule-based model, see also Section \ref{sec:ex_inter}. They generate very simple models with an interpretability index \eqref{eq:inter} much lower than the other algorithms. Moreover the sets of selected rules $\C_n$ form in most cases a covering of the training set.  CA achieves a good interpretability-accuracy trade-off with emphasis on interpretability except for the dataset {\it Student} that might be too complex  to be interpretable. SIRUS appears as a good challenger achieving a more balanced trade-off. Finally, these results show that the choice of the rule generator for the CA has an impact on the interpretability index and on the accuracy.

\subsubsection{Examples of interpretation}\label{sec:ex_inter}
Any model generated by a CA can be summarized in a table of rules. We present one set of rules selected by CA RF for the dataset \emph{Ozone} in Table~\ref{tab:rules_ozone}. The measure $\Delta_n$ is the empirical mean deviation ratio of the prediction from the mean of $Y$.

\begin{table}[!ht]
	\centering
		\begin{tabular}{|l||c|c|c|c|c|}
			\hline
			Rule & Conditions & Coverage & Prediction &  Std & $\Delta_n$\\
			\hline
			$R1$ & \makecell{$temp \in [30, 67.5]$} & $0.65$ & $7.99$ & $4.57$ & $-0.2$\\
			\hline
			$R2$ & \makecell{$ibt \in [-15, 191]$} & $0.60$ & $8.02$ & $4.5$ & $-0.2$\\
			\hline
			$R3$ & \makecell{$ibt \in [172, 326]$ \\ $ humidity \in [35.5, 92]$} & $0.41$ & $18.61$ & $7$ & $0.8$\\
			\hline
			$R4$ &  \makecell{$ibt \in [172, 326]$ \\ $ humidity \in [19, 35.5]$} & $0.08$ & $6.11$ & $2.64$ & $-0.4$\\
			\hline
	\end{tabular}
\caption{\label{tab:rules_ozone}Summary of the rules selected by CA RF for the dataset \emph{Ozone}. All rules are significant. The $temp$ variable is the Sandburg Air Force base temperature in degrees Fahrenheit, the  $ibt$ variable is the inversion base temperature at LAX in degrees Fahrenheit and the $humidity$ variable is the humidity in percent at LAX.}
\end{table}


Table~\ref{tab:rules_ozone} is a description of the model. Together with Table~\ref{tab:resum_variables} we are able to translate in natural language the significant rules in Table~\ref{tab:rules_ozone}, according to the definition of interpretability of \cite{Biran17}. 

\begin{table}[!ht]
	\centering
		\begin{tabular}{|l||c|c|c|c|}
			\hline
			{} & Y & $temp$ & $ibt$ & $humidity$\\
			\hline
			mean & 11.78 & 61.75 & 161.16 & 58.13\\
			std & 8.01 & 14.46 & 76.68 & 19.87\\
			min & 1.00 & 25.00 & -25.00 & 19.00\\
			25\% & 5.00 & 51.00 & 107.00 & 47.00\\
			50\% & 10.00 & 62.00 & 167.50 & 64.00\\
			75\% & 17.00 & 72.00 & 214.00 & 73.00\\
			max & 38.00 & 93.00 & 332.00 & 93.00\\
			\hline
		\end{tabular}
\caption{\label{tab:resum_variables}Description of the variables selected by CA RF for the dataset \emph{Ozone}. The $Y$ variable is the daily maximum of the hourly average ozone concentration in Upland.}
\end{table}

\begin{itemize}
\item Rule $R1$ has a length equal to $1$ and suggests that if the $temp$ variable is reasonably low then the daily maximum of the hourly average ozone concentrations in Upland tends to be low.
\item Rule $R2$ has a length equal to $1$ and suggests that if the $ibt$ variable is not too high then the target tends to be low.
\item Rules $R3$ and $R4$ have a length equal to $2$ and suggest that CA has detected an interaction between high $ibt$ and the $humidity$ variable. Indeed, high value of $ibt$ and high $humidity$ indicate a high hourly average ozone concentrations and high value of $ibt$ and low $humidity$ indicate a low hourly average ozone concentrations.
\end{itemize}

\subsubsection{Comments}\label{sec:comments}

In Figure~\ref{fig:mse_machine} and Figure~\ref{fig:mse_ozone} we present boxplots of the MSE over the 20 executions for each algorithm for the \emph{Machine} and \emph{Ozone} datasets. These two datasets has been chosen to show that CA can have both a very variable accuracy (Figure~\ref{fig:mse_machine}) and a stable and a good accuracy (Figure~\ref{fig:mse_ozone}) for two datasets with the same dimension.

\begin{figure}[!ht]
		\includegraphics[scale=0.16]{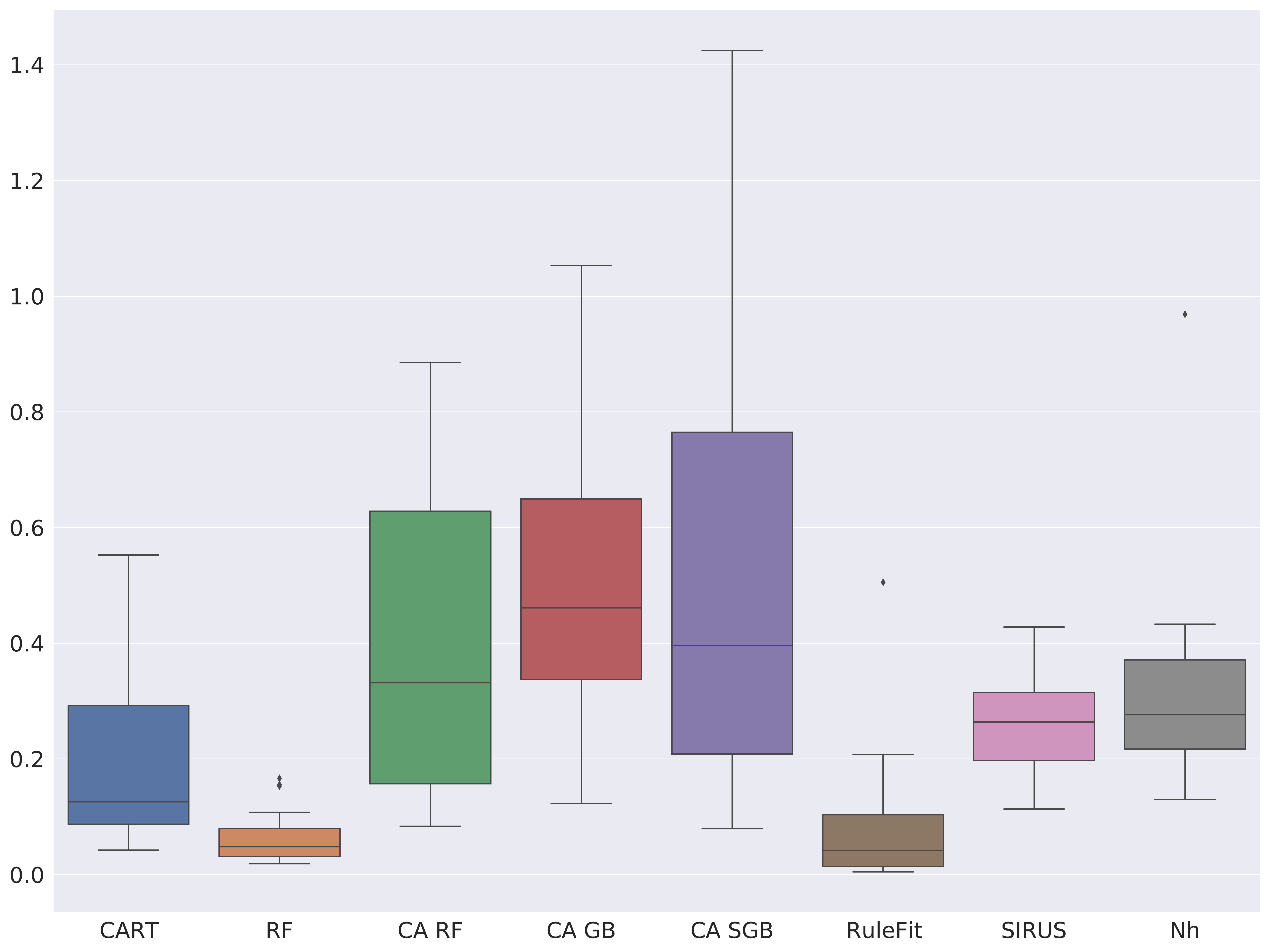}
		\caption{\label{fig:mse_machine} Boxplots of the ratio mean squared error/empirical variance of the target (MSE) over  20  executions of   algorithms for the \emph{Machine} dataset.}
\end{figure}
\begin{figure}[!ht]
		\includegraphics[scale=0.16]{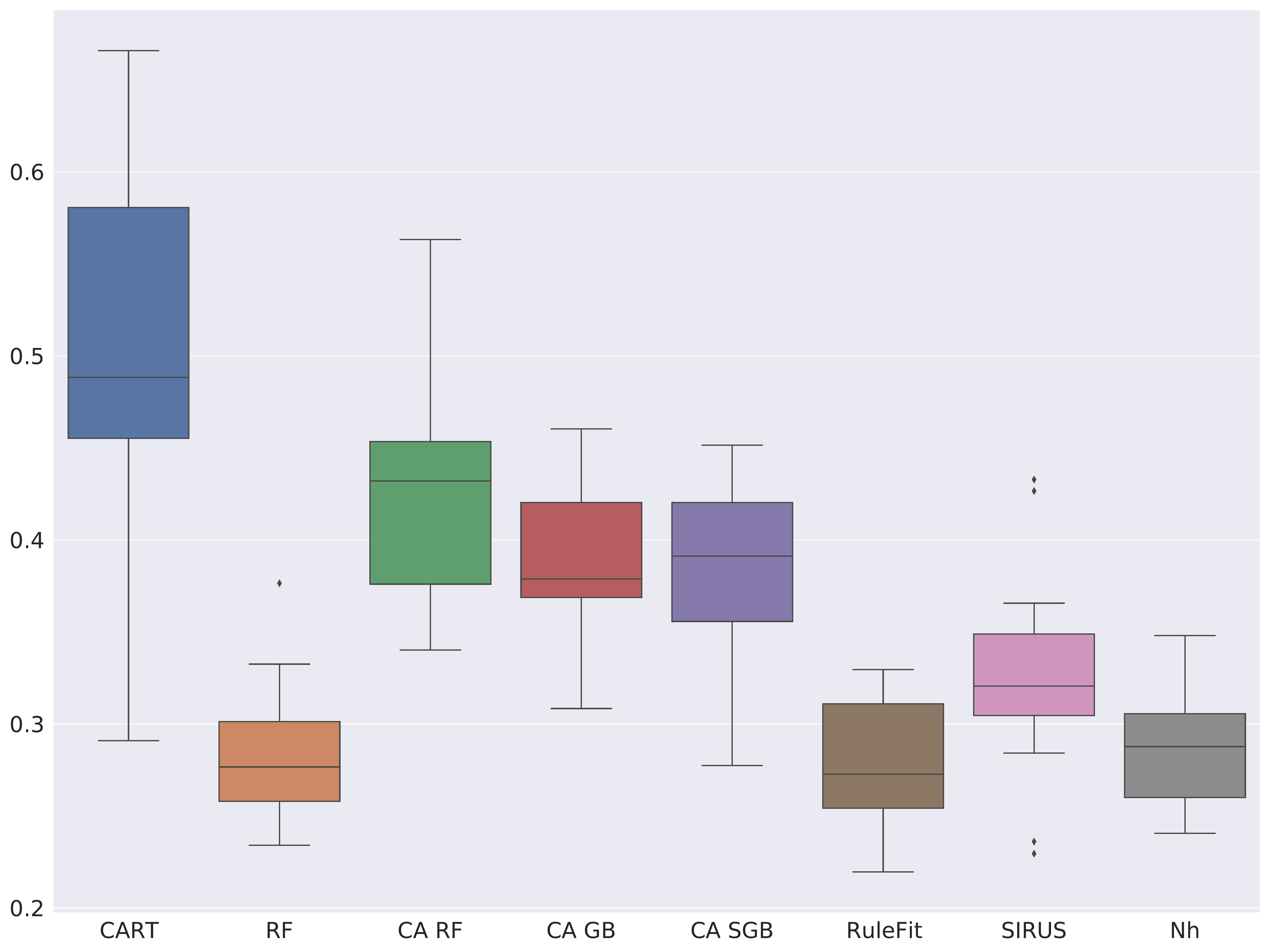}
		\caption{\label{fig:mse_ozone} Boxplots of the ratio mean squared error/empirical variance of the target (MSE) over  20  executions of algorithms for the \emph{Ozone} dataset.}
\end{figure}

One reason could be that rule generators of CA are random generators of rules that may not capture the good features at every run.  In particular we have identified the problem of predicting the empty cells of the  partition generated by the quasi-coverings. Then, by convention, $g_n(\x)=0$ for $\x \in A$ with $\Q_n(A)=0$. Thus, the accuracy of CA strongly depends on the coverage $\Q_n(A)$ of the sets $A\in \P(\C_n)$ which can be too small to be stable.  In Definition \ref{def:suitable_covering} of suitable data-dependent quasi-coverings only the coverage of the elements of $\C_n$ is controlled. The problem can be solved by considering a deterministic rule generator.

\section{Conclusion and perspectives}
In this paper, we provide a general framework for studying the consistency of rule-based interpretable estimators. We introduce the definition of a suitable quasi-covering. It is composed of two types of sets, namely the significant sets and the insignificant sets. The significant sets are considered as interpretable sets by construction. The insignificant ones are sets whose variance tends to zero. We provide a Covering Algorithm that extracts a suitable data-dependent quasi-covering from any rule generator. 

In Section \ref{sec:artif}, we run a Monte Carlo experiment on Covering Algorithm applied to Random Forest. We compare its results with those of Random Forest \citep{Breiman01} and RuleFit \citep{Friedman08}. This experiment shows that Covering Algorithm, which strives for interpretability, also identifies the support of the regression function.

In Section \ref{sec:real}, we apply Covering Algorithm to Random Forest \citep{Breiman01}, Gradient Boosting \citep{Friedman01} and Stochastic Gradient Boosting \citep{Friedman02} and we compare their results with those of CART \citep{CART}, Random Forest \citep{Breiman01}, RuleFit \citep{Friedman08}, Node harvest \citep{Meinshausen10} and SIRUS \citep{Benard20}. The loss of accuracy in the prediction is the cost of having an interpretable model according to our definition of interpretability. We broaden the accuracy-interpretability trade-off of classical algorithms by providing a much more interpretable method that remains consistent. 

Our methodology based on quasi-coverings is very effective in generating interpretable models. The use of tree ensembles such as RF, GB or SGB as rule generator is questionable; In Section \ref{sec:comments}, we pointed out the possible negative effect of the randomization procedure in combination with the ERM principle instead of averaging. We also noted that the choice of the rule generator has an important effect on accuracy, interpretability and stability. We are now looking for an algorithm that satisfies \ref{H:partitions} and that deterministically generates significant and insignificant rules that form a suitable sequence of data-dependent quasi-coverings.

In practice the variance $\sigma^2$ is unknown and has to be estimated efficiently enough.  
We let for future work the difficult question of defining a rule-based estimator of the variance with rate of convergence $O_\Pr(n^{\alpha-1/2})$ as required in Theorem \ref{th:consistance}. Our setting could also be broadened; unbounded $Y$ may be considered by introducing a truncation operator as in \cite{DistributionFree}; strong consistency and rates of convergence of the data-dependent covering estimators may be established under regularity conditions on $g^\ast$. Finally, the scope could be easily adapted from the regression setting to the classification setting by adapting the significant condition accordingly.

\section*{Appendix}

We gather here some proofs and provide the pseudo-code of the Covering Algorithm's selection process.\\

\subsection*{Proof of Proposition~\ref{lemma:PC}}	
	\begin{proof}
		By the definition of $\varphi_\C$, for any $\x\in\Rset^d$, $\varphi_\C^{-1}(\varphi_C(\x)) = \bigcap_{\r \in \C :  \x \in \r} \r \setminus \bigcup_{\r \in \C :  \x \notin \r} \r$.
		Thus
		\begin{align*}
			A\in\P(\C) = \varphi_\C^{-1}(Im (\varphi_\C)) &\Longleftrightarrow \exists \x\in\Rset^d / A = \varphi_\C^{-1}(\varphi_C(\x)) = \mathop{\bigcap_{\r \in \C}}_{\x \in \r} \r \setminus \mathop{\bigcup_{\r \in \C}}_{\x \notin \r} \r \\
			&\Longleftrightarrow \exists \tilde\C \subseteq \C / A =  \bigcap_{\r \in \tilde\C} \r \setminus \bigcup_{\r \in \C\setminus\tilde\C} \r \text{ and } A \neq \emptyset.
		\end{align*}
	\end{proof}

\subsection*{Proof of Proposition~\ref{prop:compPartitionCovering}}	

\begin{proof}
\begin{enumerate}
\item
\begin{itemize}
\item Lower bound: Consider the set $\mathcal E_1$ of the points $x\in\Rset^d$ whose coordinates are all zero except $x[i]=2$, for $1\leq i\leq d$ and  the set $\mathcal E_2$ of the points $x\in\Rset^d$ whose coordinates are all zero except $x[i]=1$ and $x[i+1]=-1$, for $1\leq i\leq d$ (with the convention $x[d+1]=x[1]$).

If $\P$ is a partition by hyperrectangles and $C\in\P$, then no pair of the $2d$ points of $\mathcal E_1\cup\mathcal E_2$ can be in the same element of $\P\setminus \{C\}$.

Indeed, suppose that $(x,y)\in\mathcal E_1^2$ (or analogously $\in\mathcal E_1^2$) lies in some $D\in\P\setminus \{C\}$, with $x[i]=2$ and $y[j]=2$. Since $R$ is an hyperrectangle, the point $z$ whose coordinates are all zero except $z[i]=\frac{x[i]+y[i]}{2}=1$ and $z[j]=\frac{x[j]+y[j]}{2}=1$ is in $R$ too. Thus $C$ and $R$ intersect. In the same way, suppose that $(x,y)\in\mathcal E_1\times\mathcal E_2$ lies in some $D\in\P\setminus \{C\}$, with $x[i]=2$, $y[j]=1$ and $y[j+1]=-1$. Then, the point $z$ whose coordinates are all zero except $z[i]=1$ (while $z[j]=x[j]=0$ and $z[j+1]=x[j+1]=0$) is in $R$ too. Thus $C$ and $R$ intersect.

Thus, the $2d$ points of $\mathcal E_1\cup\mathcal E_2$ induce $2d$ distinct elements of $\P$ that are added to $C$. 
\item Upper bound: The partition 
	\begin{align*}
		\P = &\bigl\{(-\infty,0)\, \times\, \Rset^{d-1} \,;\, (1,+\infty)\, \times\, \Rset^{d-1} \,; \\ 
		&\qquad [0,1] \times (-\infty,0)\, \times\, \Rset^{d-2} \,;\, [0,1] \times (1,+\infty) \,\times\, \Rset^{d-2} \,; \\
		&\qquad\quad \dots\,; \\
		&\qquad\quad\quad [0,1]^{d-1} \,\times\, (-\infty,0) \;; [0,1]^{d-1} \,\times\, (1,+\infty) \;; \\
		&\qquad\quad\quad\quad [0,1]^d \bigr\}
	\end{align*}
	 fulfills the conditions of the proposition.
\end{itemize}
\item The covering $\C = \bigl\{ [0,1]^d \,;\, \Rset^d \bigr\}$	fulfills the conditions of the proposition.
\end{enumerate}
\end{proof}

\subsection*{Proof of Proposition~\ref{prop:emp}}
For any $f : \S \rightarrow \mathbb R$ in $\L ^1(\Q)$ we note classically $\Q f := \int f d\Q$, $\Q_n f := \int f d\Q_n$.

\begin{proof}
	Let $\varepsilon >0$. 
	First, for any $f\in\F$ and $A\in\Br_n$, since $\Q_n(A) > 0$ and then $\Q(A) > 0$,
	\begin{align}
	& \left| \E_n \left[ f \mid A \right] - \E \left[f \mid A \right] \right| \notag\\
	&= \left| \frac{\int_{A} f d\Q_n}{\Q_n(A)} - \frac{\int_{A} f d\Q}{\Q(A)} \right| \notag\\
	&= \left| \frac{\Q(A) \left( \int_{A} f d\Q_n - \int_{A} f d\Q \right) + \left(\Q(A) - \Q_n(A)\right) \int_{A} f d\Q }{ \Q(A)\Q_n(A) }\right| \notag\\
	&\leq \left| \frac{\int_{A} f d\Q_n - \int_{A} f d\Q }{ \Q_n(A)} \right| + \left| \left( \Q(A) - \Q_n(A)\right) \frac{\int_{A} f d\Q }{ \Q(A)\Q_n(A)} \right|. \label{ineg:EspEmpTh}
	\end{align}
	
	Now, according to Proposition~\ref{prop:donsker},
	\begin{equation*}
	\sup \limits_{\tilde f \in \F, \tilde A \in B} \left| \int_{\tilde A} \tilde f d\Q_n - \int_{\tilde A} \tilde f d\Q \right| = O_{\Pr^*}(n^{-1/2}),
	\end{equation*}
	and
	\begin{equation*}
	\sup \limits_{\tilde A \in \Br} \left| \Q_n(\tilde A) - \Q(\tilde A) \right| = O_{\Pr^*}(n^{-1/2}). 
	\end{equation*}
	
	Thus, According to Remark~\ref{rem:as-tight-R}, there exists $M>0$ such that for any $n$ large enough,
	
	\begin{equation*}
	\Pr^* \left\{ \sup \limits_{\tilde f \in \F, \tilde A \in B} \left| \int_{\tilde A} \tilde f d\Q_n - \int_{\tilde A} \tilde f d\Q \right| > Mn^{-1/2} \right\} < \frac\varepsilon 2,
	\end{equation*} 
	and
	\begin{equation*}
	\Pr^* \left\{ \sup \limits_{\tilde A \in \Br} \left| \Q_n(\tilde A) - \Q(\tilde A) \right| > Mn^{-1/2} \right\} < \frac\varepsilon 2,
	\end{equation*} 
	so that $\Pr^*(\Omega_n) \geq 1 - \varepsilon$ with 	
	\begin{multline*}\label{def_omegan}
	\Omega_n := \Bigl\{ \sup \limits_{\tilde f \in \F, \tilde A \in B} \Bigl| \int_{\tilde A} \tilde f d\Q_n - \int_{\tilde A} \tilde f d\Q \Bigr| \leq Mn^{-1/2} \Bigr\} \bigcap \\ \Bigl\{ \sup \limits_{\tilde A \in B} \bigl| \Q_n(\tilde A) - \Q(\tilde A) \bigr| \leq Mn^{-1/2} \Bigr\}.
	\end{multline*}
	
	Then \eqref{ineg:EspEmpTh} yields, with $c := \sup \limits_{f \in \F, \x \in \S} |f(\x)| < \infty$ and since $\Q_n(A) \geq n^{-\alpha}$, for $n$ large enough, in the event $\Omega_n$,
	\begin{equation*}
	\sup \limits_{f \in \F, A \in \Br_n} \bigl| \E_n \left[ f \mid A \right] - \E \left[f \mid A \right] \bigr| \leq Mn^{\alpha-1/2} (1+c),
	\end{equation*}
	since $\frac{\int_{A} f d\Q}{\Q(A)} \leq c$.
	
	Finally, it has been proved that $\forall \varepsilon > 0, \exists M > 0, \exists N \in \mathbb N^* / \forall n \geq N,$
	\begin{equation*}
	\Pr^* \left\{ \sup \limits_{f \in \F, A \in \Br_n} \left| \E_n \left[ f \mid A \right] - \E \left[f \mid A \right] \right| > Mn^{\alpha-1/2} \right\} < \varepsilon,
	\end{equation*}
	and then $\forall \varepsilon > 0, \exists M > 0$ such that
	\begin{equation*}
	\limsup_{n\rightarrow\infty} \Pr^* \left\{ \sup \limits_{f \in \F, A \in \Br_n} \left| \E_n \left[ f \mid A \right] - \E \left[f \mid A \right] \right| > Mn^{\alpha-1/2} \right\} \leq \varepsilon,
	\end{equation*}
	which, together with Remark~\ref{rem:as-tight-R} again, proves the proposition. 
\end{proof}

\subsection*{Proof of Corollary \ref{prop:expectation} }
\begin{proof}[Proof of \eqref{eq:expectation}]
	Let $L=\essup Y$, $i\in\Nset,$ and $f_i\in \L^1(\Q)$ be defined by
	\begin{align*}
	f_i : \Rset^d \times [-L,L] & \to [-L^i,L^i] \\
	(\x, y) &\mapsto y^i.
	\end{align*}
	$f_i$ is bounded and $\{f_i\}$ is finite thus Donsker.	
	The result is then a straightforward application of Proposition~\ref{prop:emp}.
\end{proof}
\begin{proof}[Proof of \eqref{eq:variance}] 
	This part follows from Proposition~\ref{prop:emp} since $Y$ is bounded and \ $$\V_n\left[ Y \mid (\X , Y) \in A \right] := \E_n \left[ Y^2 \mid (\X , Y) \in A \right] - \E_n \left[ Y \mid (\X , Y) \in A \right]^2.$$
\end{proof}

\clearpage

\subsection*{Covering Algorithm selection process}
\begin{algorithm}
	\caption{Selection of minimal set of rules}
	\label{pseudocode}
	\SetAlgoLined
	\KwIn
	{
		\begin{itemize}
			\item the rate $0<\gamma<1$;
			\item a set of significant rules $S$;
			\item a set of insignificant rules $I$;
		\end{itemize}
	}
	\KwOut
	{
		\begin{itemize}\vspace{-0.2cm}
			\item a minimal set of rules $\C_n$;
		\end{itemize}
	}
	$\C_n \leftarrow \argmax_{\r \in S} \Q_n(\r)$\;
	$S \leftarrow S \setminus \C_n$\;
	\While{$\sum_{\r \in \C_n} \Q_n(\r) < 1$}
	{
		$\r^* \leftarrow \argmax_{\r \in S} \Q_n(\r)$\;
		\uIf{ $\Q_n(\r^* \cap \{\cup_{\r \in \C_n} \r\}) \le \gamma\,\Q_n(\r^*)$}
		{
			$\C_n \leftarrow \C_n \cup \r^*$\;
		}
		
		$S \leftarrow S \setminus \r^*$\;
		\uIf{$\#S = 0$}
		{
			Break \;
		}
	}

	\While{$\sum_{\r \in \C_n} \Q_n(\r) < 1$}
	{
		$\r^* \leftarrow \argmin_{\r \in I} \V_n(Y | \X \in \r)$\;
		\uIf{ $\Q_n(\r^* \cap \{\cup_{\r \in \C_n} \r\}) \le \gamma\,\Q_n(\r^*)$}
		{
			$\C_n \leftarrow \C_n \cup \r^*$\;
		}
			
		$I \leftarrow I \setminus \r^*$\;
		\uIf{$\#I = 0$}
		{
			Break \;
		}
	}
	\Return{} $\C_n$\;	
\end{algorithm}


\bibliographystyle{apalike}
\bibliography{biblio} 

\begin{thebibliography}{}

\bibitem[B{\'e}nard et~al., 2020]{Benard20}
B{\'e}nard, C., Biau, G., Da~Veiga, S., and Scornet, E. (2020).
\newblock Interpretable random forests via rule extraction.
\newblock {\em arXiv preprint arXiv:2004.14841}.

\bibitem[Biran and Cotton, 2017]{Biran17}
Biran, O. and Cotton, C. (2017).
\newblock Explanation and justification in machine learning: A survey.
\newblock In {\em IJCAI-17 workshop on explainable AI (XAI)}, volume~8, page~1.

\bibitem[Breiman, 2001]{Breiman01}
Breiman, L. (2001).
\newblock Random forests.
\newblock {\em Machine learning}, 45(1):5--32.

\bibitem[Breiman et~al., 1984]{CART}
Breiman, L., Friedman, J., Olshen, R., and Stone, C. (1984).
\newblock {\em Classification and Regression Trees.}
\newblock CRC press.

\bibitem[Cortez and Silva, 2008]{Cortez08}
Cortez, P. and Silva, A. M.~G. (2008).
\newblock Using data mining to predict secondary school student performance.
\newblock In {\em Proceedings of 5th FUture BUsiness TEChnology Conference
  (FUBUTEC 2008)}.

\bibitem[Denil et~al., 2013]{Denil13}
Denil, M., Matheson, D., and Freitas, N. (2013).
\newblock Consistency of online random forests.
\newblock In {\em International Conference on Machine Learning}, pages
  1256--1264.

\bibitem[Dua and Graff, 2017]{Dua2019}
Dua, D. and Graff, C. (2017).
\newblock Uci machine learning repository.

\bibitem[Efron et~al., 2004]{Efron04}
Efron, B., Hastie, T., Johnstone, I., Tibshirani, R., et~al. (2004).
\newblock Least angle regression.
\newblock {\em The Annals of statistics}, 32(2):407--499.

\bibitem[{Friedman} and {Popescu}, 2008]{Friedman08}
{Friedman}, J. and {Popescu}, B. (2008).
\newblock Predective learning via rule ensembles.
\newblock {\em The Annals of Applied Statistics}, pages 916--954.

\bibitem[Friedman et~al., 2003]{Friedman03}
Friedman, J., Popescu, B., et~al. (2003).
\newblock Importance sampled learning ensembles.
\newblock {\em Journal of Machine Learning Research}, 94305.

\bibitem[Friedman, 2001]{Friedman01}
Friedman, J.~H. (2001).
\newblock Greedy function approximation: a gradient boosting machine.
\newblock {\em Annals of statistics}, pages 1189--1232.

\bibitem[Friedman, 2002]{Friedman02}
Friedman, J.~H. (2002).
\newblock Stochastic gradient boosting.
\newblock {\em Computational statistics \& data analysis}, 38(4):367--378.

\bibitem[F{\"u}rnkranz and Kliegr, 2015]{Furnkranz15}
F{\"u}rnkranz, J. and Kliegr, T. (2015).
\newblock A brief overview of rule learning.
\newblock In {\em International Symposium on Rules and Rule Markup Languages
  for the Semantic Web}, pages 54--69. Springer.

\bibitem[Grunewalder, 2018]{Grunewalder18}
Grunewalder, S. (2018).
\newblock Plug-in estimators for conditional expectations and probabilities.
\newblock In {\em International Conference on Artificial Intelligence and
  Statistics}, pages 1513--1521.

\bibitem[Guidotti et~al., 2018]{Guidotti18}
Guidotti, R., Monreale, A., Ruggieri, S., Turini, F., Giannotti, F., and
  Pedreschi, D. (2018).
\newblock A survey of methods for explaining black box models.
\newblock {\em ACM computing surveys (CSUR)}, 51(5):1--42.

\bibitem[{Gy\"{o}rfi} et~al., 2006]{DistributionFree}
{Gy\"{o}rfi}, L., {Kohler}, M., {Krzyzak}, A., and {Walk}, H. (2006).
\newblock {\em A Distribution-Free Theory of Nonparametric Regression}.
\newblock Springer Science \& Business Media.

\bibitem[Harrison~Jr and Rubinfeld, 1978]{Harrison78}
Harrison~Jr, D. and Rubinfeld, D.~L. (1978).
\newblock Hedonic housing prices and the demand for clean air.
\newblock {\em Journal of environmental economics and management},
  5(1):81--102.

\bibitem[Hastie et~al., 2001]{ElementST}
Hastie, T., Friedman, J., and Tibshirani, R. (2001).
\newblock {\em The Elements of Statistical Learning}, volume~1.
\newblock Springer series in statistics Springer, Berlin.

\bibitem[Holmes et~al., 1999]{Holmes99}
Holmes, G., Hall, M., and Prank, E. (1999).
\newblock Generating rule sets from model trees.
\newblock In {\em Australasian Joint Conference on Artificial Intelligence},
  pages 1--12. Springer.

\bibitem[Karali{\v{c}} and Bratko, 1997]{Karalivc97}
Karali{\v{c}}, A. and Bratko, I. (1997).
\newblock First order regression.
\newblock {\em Machine Learning}, 26(2-3):147--176.

\bibitem[Liiti{\"a}inen et~al., 2009]{Liitiainen09}
Liiti{\"a}inen, E., Verleysen, M., Corona, F., and Lendasse, A. (2009).
\newblock Residual variance estimation in machine learning.
\newblock {\em Neurocomputing}, 72(16-18):3692--3703.

\bibitem[Lipton, 2018]{Lipton18}
Lipton, Z.~C. (2018).
\newblock The mythos of model interpretability.
\newblock {\em Queue}, 16(3):31--57.

\bibitem[Lundberg and Lee, 2017]{Lundberg17}
Lundberg, S.~M. and Lee, S.-I. (2017).
\newblock A unified approach to interpreting model predictions.
\newblock In {\em Advances in Neural Information Processing Systems}, pages
  4765--4774.

\bibitem[Margot et~al., 2018]{Margot18}
Margot, V., Baudry, J.-P., Guilloux, F., and Wintenberger, O. (2018).
\newblock Rule induction partitioning estimator.
\newblock In {\em International Conference on Machine Learning and Data Mining
  in Pattern Recognition}, pages 288--301. Springer.

\bibitem[Meinshausen, 2010]{Meinshausen10}
Meinshausen, N. (2010).
\newblock Node harvest.
\newblock {\em The Annals of Applied Statistics}, 4(4):2049--2072.

\bibitem[{Nobel}, 1996]{Nobel96}
{Nobel}, A. (1996).
\newblock Histogram regression estimation using data-dependent partitions.
\newblock {\em The Annals of Statistics}, 24(3):1084--1105.

\bibitem[Quinlan, 1986]{Quinlan86}
Quinlan, J.~R. (1986).
\newblock Induction of decision trees.
\newblock {\em Machine learning}, 1(1):81--106.

\bibitem[Quinlan, 1993]{Quinlan93}
Quinlan, J.~R. (1993).
\newblock C4.5: Programs for machine learning.

\bibitem[Ramosaj and Pauly, 2019]{ramosaj19}
Ramosaj, B. and Pauly, M. (2019).
\newblock Consistent estimation of residual variance with random forest
  out-of-bag errors.
\newblock {\em Statistics \& Probability Letters}, 151:49--57.

\bibitem[Ribeiro et~al., 2016]{Ribeiro16}
Ribeiro, M.~T., Singh, S., and Guestrin, C. (2016).
\newblock Why should i trust you?: Explaining the predictions of any
  classifier.
\newblock In {\em Proceedings of the 22nd ACM SIGKDD international conference
  on knowledge discovery and data mining}, pages 1135--1144. ACM.

\bibitem[Scornet et~al., 2015]{Scornet15}
Scornet, E., Biau, G., Vert, J.-P., et~al. (2015).
\newblock Consistency of random forests.
\newblock {\em The Annals of Statistics}, 43(4):1716--1741.

\bibitem[Shrikumar et~al., 2019]{Shrikumar19}
Shrikumar, A., Greenside, P., and Kundaje, A. (2019).
\newblock Learning important features through propagating activation
  differences.
\newblock {\em arXiv preprint arXiv:1704.02685}.

\bibitem[Tibshirani, 1996]{Tibshirani96}
Tibshirani, R. (1996).
\newblock Regression shrinkage and selection via the lasso.
\newblock {\em Journal of the Royal Statistical Society. Series B
  (Methodological)}, pages 267--288.

\bibitem[Van~der Vaart, 2000]{VanDerVaart00}
Van~der Vaart, A.~W. (2000).
\newblock {\em Asymptotic statistics}, volume~3.
\newblock Cambridge university press.

\bibitem[Vapnik, 1995]{Vapnikbook}
Vapnik, V.~N. (1995).
\newblock {\em The Nature of Statistical Learning Theory}.
\newblock Springer-Verlag.

\bibitem[Wenocur and Dudley, 1981]{wenocur1981}
Wenocur, R.~S. and Dudley, R.~M. (1981).
\newblock Some special {V}apnik-{C}hervonenkis classes.
\newblock {\em Discrete Mathematics}, 33(3):313--318.

\bibitem[Zhao and Bhowmick, 2003]{Zhao03}
Zhao, Q. and Bhowmick, S.~S. (2003).
\newblock Association rule mining: A survey.
\newblock {\em Nanyang Technological University, Singapore}.

\end{thebibliography}

\end{document}